\documentclass[11pt]{amsart}

\usepackage[colorlinks,bookmarksopen,bookmarksnumbered,citecolor=red,urlcolor=red]{hyperref}
\usepackage{graphicx}
\usepackage{bm,color}
\usepackage[section]{algorithm}
\usepackage{algorithmic}
\usepackage{enumerate}
\usepackage{amsmath}
\usepackage{amsfonts}
\usepackage{amssymb}
\usepackage{amsthm}
\usepackage{verbatim}
\usepackage{multirow}

\def\bWt{{\mathbb{W}_\tau}}
\def\bVh{{\mathbb{X}_h}}
\def\bVxh{{\mathbb{V}_h}}
\def\bVht{{\mathbb{V}_{h \tau}}}
\def\bVt{{\mathbb{U}_\tau}}
\def\bBal{{\mathbb{B}_{h \tau}^\alpha}}
\def\bBalh{{\mathbb{B}_{h}^\alpha}}
\def\al{{\alpha}}
\def\Bs{{B^s}}
\def\Bal{{B^\alpha}}

\def\tribar{\vert\thickspace\!\!\vert\thickspace\!\!\vert}

\theoremstyle{plain}
\newtheorem{theorem}{Theorem}[section]
\newtheorem{lemma}{Lemma}[section]

\theoremstyle{theorem}

\newtheorem{remark}{Remark}[section]
\numberwithin{equation}{section}

\begin{document}
	
\title[Space-Time FEM for Fractional Diffusion Problems]{Space-Time Petrov-Galerkin FEM for Fractional Diffusion Problems}
\author{Beiping Duan}
\address{School of Mathematics and Statistics, Central South University, 410083 Changsha, P.R. China and
Department of Mathematics, Texas A\&M University, College Station, TX 77843-3368, USA}
\email{duanbeiping@hotmail.com}
\author{Bangti Jin}
\address{Department of Computer Science, University College London, Gower Street, London WC1E 6BT, UK}
\email{bangti.jin@gmail.com, b.jin@ucl.ac.uk}
\author{Raytcho Lazarov}
\address{Department of Mathematics, Texas A\&M University, College Station, TX 77843-3368, USA and
Institute of Mathematics and Informatics, Bulgarian Academy of Sciences, Acad. Georgi Bonchev str., block 8,  1113 Sofia, Bulgaria}
\email{lazarov@math.tamu.edu, raytcho.lazarov@gmail.com}
\author{Joseph Pasciak}
\address{Department of Mathematics, Texas A\&M University, College Station, TX 77843-3368, USA}
\email{pasciak@math.tamu.edu}
\author{Zhi Zhou}
\address{Department of Applied Physics and Applied Mathematics,
Columbia University, 500 W. 120th Street, New York, NY 10027, USA}
\email{zhizhou0125@gmail.com}
	
\subjclass[2000]{Primary 65M60, 65M15}

\date{started May 9, 2016; today is \today}

\begin{abstract}
We present and analyze a space-time Petrov-Galerkin finite element method for a time-fractional diffusion equation
involving a Riemann-Liouville fractional derivative of order $\alpha\in(0,1)$ in time and zero initial data.
We derive a proper weak formulation
involving different solution and test spaces and show the inf-sup condition for the bilinear form and thus its well-posedness.
Further, we develop a novel finite element formulation, show the well-posedness of the discrete problem, and
establish error bounds in both energy and $L^2$ norms for the finite element solution. In the proof
of the discrete inf-sup condition, a certain nonstandard $L^2$ stability property of the $L^2$ projection operator plays a key role.
We provide extensive numerical examples to verify the convergence of the method.
\end{abstract}
\keywords{space-time finite element method, Petrov-Galerkin method, fractional diffusion, error estimates}

\maketitle

\section{Introduction}\label{sec:intro}

In this work we develop and analyze a novel space-time Petrov-Galerkin formulation for time-fractional diffusion.
Let $\Omega\subset \mathbb{R}^d$ ($d=1,2,3$) be a bounded convex domain with a polygonal boundary $\partial\Omega$.
Consider the following initial boundary value problem for the function $u(x,t)$:
\begin{equation}\label{orig}
\begin{aligned}
{_0\partial_t^\alpha} u -\Delta  u &= f,\quad\mbox{in } Q_T:=\Omega \times [0,T],\\
u(x,t)&=0, \quad \mbox{on } \partial \Omega\times (0,T],\\
u(x,0)&=0 ,\quad \mbox{in } \Omega,
\end{aligned}
\end{equation}
where $f$ is a given source term, and $T>0$ is a given final time.
Here $_0\partial_t^\alpha u$ denotes the left-sided Riemann-Liouville fractional derivative of order
$\alpha \in (0,1)$ in $t$, cf. \eqref{eqn:RiemannCaputo} below.

The interest in the  model \eqref{orig} is motivated by fractional calculus
and its numerous applications related to anomalous diffusion, e.g., underground flow, thermal diffusion in fractal domains,
dynamics of protein molecules, and heat conduction with memory, to name just a few. At a microscopic level, anomalously slow
diffusion (also known as subdiffusion) processes can be described by continuous time
random walk with a heavy-tailed waiting time distribution, and
the corresponding macroscopic model is a diffusion equation with a fractional-order derivative in time, cf. (\ref{orig}).
We refer interested readers to \cite{metzler2014anomalous} for a comprehensive overview of various mathematical models,
physical backgrounds, and an extensive list of applications in physics, engineering and biology.

For standard parabolic problems, it is customary to apply time-stepping schemes \cite{Thomee:2006}.
However, space-time discretizations have gained some popularity  in the last decade.
These studies are mostly motivated by the goal to obtain efficient and convergent numerical methods
without any regularity assumptions \cite{DiPietroErn2010,RiviereGirault2016} or to design
efficient space-time adaptive algorithms
\cite{Maubach_1989,ewing1990finite,NeumullerSteinbach:2011,SchwabStevenson:2009} and
high-order schemes for general parabolic equations \cite{bank2017arbitrary}.
In the past decades, time stepping methods have also been very popular for problems involving
fractional derivatives in time
(see e.g., \cite{JinLazarovZhou:2016sisc,JinLiZhou:2016ima,Lubich:1988,MustaphaAbadallah:2014}
and references therein). 
However, due to the non-locality of the fractional derivative $_0\partial_t^\alpha u$, at each time step
one has to use the numerical solutions at all preceding time levels. Thus,
the advantages of time stepping schemes, compared to space-time schemes,
are not as pronounced as in the case of standard parabolic problems, and it is
natural to consider time-space discretization.

In this work we present a space-time variational (weak) formulation for problem \eqref{orig} and show
an inf-sup condition in Lemma \ref{lem:inf-sup}. Starting from the weak form we develop a novel discretization
that is based on tensor product meshes in time and space. The spatial domain $\Omega$ is discretized by
a quasi-uniform triangulation with a mesh size $h$, while in time by a
uniform mesh with step-size $\tau$. The approximation $u_{h \tau}$
is sought in the tensor product space $ \mathbb{X}_h \otimes \bVt$,
where $ \mathbb{X}_h $  is the space of continuous piecewise linear
functions in the spacial variable $x$ and  $ \bVt$
is the space of fractionalized piecewise constant functions in the time variable $t$. The test space
is a tensor product space $ \mathbb{X}_h \otimes \bWt$,
where $ \bWt $ is the space of piecewise constant functions in time, cf. \eqref{eqn:fy}.
We establish an inf-sup condition for the discrete formulation, using  
the $L^2$-projection from $\bVt$ to $\bWt$, cf. Lemma \ref{ptau-stability}.
It is worth noting that the constant in the $L^2$-stability of of this projection
depends on the fractional order $\alpha$ and
deteriorates as $\alpha \to 1$, confirmed by our
computations in Table \ref{tab:Pi-bound}. Thus, for standard parabolic
problems ($\alpha=1$), it depends on the time step size $\tau$, leading to an
undesirable CFL-condition, a fact established in \cite{LarssonMolteni:2016}.
A distinct algorithmic feature of the proposed approach is that it leads to a time-stepping
like scheme, and thus admits an efficient practical implementation.

Optimal-order error estimates in both energy and $L^2(Q_T)$ norms are provided under suitable temporal
regularity of the source term $f$ in Theorems \ref{thm:err-energy} and \ref{thm:err-L2}.
The error analysis is carried out in two steps. First, we introduce a space semidiscrete
approximation $u_h$ and derive sharp error bounds for $u - u_h$ in both $\Bal(Q_T)$-
and $L^2(Q_T)$-norms, by applying the inf-sup condition for the semidiscrete problem
and an approximation result from \cite{JinLazarovPasciakZhou:2015}.
Second, we bound the difference $ u_h - u_{h \tau}$. This is achieved by a careful study
of the initial value problem of the fractional ODE ${_0\partial_t^\alpha} u + \lambda  u = f$, $\lambda >0$.
The uniform (with respect to $\lambda$) stability of the ODE and its optimal approximation
in the space $\bVt$ play a key role in the error analysis. 
Then by expanding $u_h(t)$ and $u_{h \tau}$ in eigenfunctions of the discrete Laplacian, and using the result
for the fractional ODE, in Theorem \ref{thm:err-L2}, we obtain the desired error estimates for
$f \in \widetilde H_L^s(0,T; L^2(\Omega))$, $ 0 \le s \le 1$. In particular, for $f \in L^2(Q_T)$, we have
\begin{equation*}
\| u - u_{h \tau}\|_{L^2(Q_T)} \le c (\tau^\alpha + h^2) \|f\|_{L^2(Q_T)}.
\end{equation*}

The rest of the paper is organized as follows. In Section \ref{sec:space-time}, we
recall preliminaries from fractional calculus, derive the
space-time variational formulation, and analyze its well-posedness and the
solution regularity pickup. In Section \ref{sec:FEM}, we develop a Petrov-Galerkin
FEM based on the variational formulation and a tensor product
mesh, establish a
discrete inf-sup condition and discuss the resulting linear algebraic formulation. The
convergence analysis is given in Sections \ref{sec:conv-ODE} and \ref{sec:conv-PDE}
for fractional ODEs and PDEs, respectively. Some numerical results that illustrate our
theoretical analysis are presented in Section \ref{sec:numerics}.

Throughout, the notation $c$, with or without a subscript, denotes a generic constant,
which may change from one line to another but which is always independent of the spatial
mesh size $h$ and time step size $\tau$.
We will use the following convention: for a function space $S$ (dependent of the
variable $t$ or/and $x$), the notations $\mathbb{S}_\tau$ and  $\mathbb{S}_h$
denote the time- and space-discrete counterpart, respectively, and $\mathbb{S}_{h\tau}$ for the
space-time discrete counterpart. \\

\section{Time-space formulation}\label{sec:space-time}
In this section we develop a space-time variational formulation,
and analyze its well-posedness.

\subsection{Notation and preliminaries}
First, we recall some preliminary facts and
notations  from fractional calculus.
For any $\gamma > 0$ and $ u \in L^2(0,T)$, we define the left-sided and right-sided Riemann-Liouville
fractional integral operators, i.e., $_0\hspace{-0.3mm}I^{\gamma}_t$ and $_tI_T^\gamma$, of order $\gamma$ respectively by
\begin{equation}\label{eqn:RL-int}
({_0I^\gamma_t} u) (t)= \frac 1{\Gamma(\gamma)} \int_0^t (t-s)^{\gamma-1} u(s)ds\quad
\mbox{and}\quad
({_tI^\gamma_T} u) (t)= \frac 1{\Gamma(\gamma)}\int_t^T (s-t)^{\gamma-1}u(s)\,ds,
\end{equation}
where $\Gamma(\cdot)$ is Euler's Gamma function defined by $\Gamma(z)=\int_0^\infty s^{z-1}e^{-s}ds$ for $\Re z>0$.

For any $\beta>0$ with $k-1 < \beta < k$, $k\in\mathbb{N}^+$, the (formal)
left-sided and right-sided Riemann-Liouville fractional derivative of order $\beta$ are respectively defined by
\begin{equation}\label{eqn:RiemannCaputo}
_0\partial_t^\beta u =\frac {d^k} {d t^k}({_0I^{k-\beta}_t} u)\quad \mbox{and} \quad
{_t\partial_T^\beta } u =(-1)^k\frac {d^k} {d t^k}({_tI^{k-\beta}_T} u),
\end{equation}
These fractional-order derivatives are well defined for sufficiently smooth functions.
	
Next we introduce the space $\widetilde{H}_L^s(0,T) $ (respectively $\widetilde{H}_R^s(0,T) $),
which consists of functions whose extension by zero belong to $H^{s}(-\infty,T)$ (respectively
$H^s(0,\infty)$) \cite{Grisvard}. We have the following useful identity  \cite[pp. 76, Lemma 2.7]{KilbasSrivastavaTrujillo:2006}
\begin{equation}\label{eqn:adjoint}
\int_0^T ({_0\partial_t^\alpha u(t)}) v(t) \ dt = \int_0^T u(t) ({_t\partial_T^\alpha v(t)})\ dt
\quad \forall u \in \widetilde{H}_L^\alpha(0,T),\, v \in  \widetilde{H}_R^\alpha(0,T).
\end{equation}

On the cylinder $Q_T=\Omega\times(0,T)$, we define the  $L^2(Q_T)$-norm
in a standard way:
\begin{equation*}
	\langle u, v \rangle_{L^2(Q_T)} = \int_0^T \int_{\Omega} uv \,  dx dt\quad \forall u,v\in L^2(Q_T).
\end{equation*}
The notation $(\cdot, \cdot)_{L^2(\Omega)}$ denotes the duality pairing between
$H^1_0(\Omega)$  and its dual $H^{-1}(\Omega)$,
also the inner product in $L^2(\Omega)$.
For functions $u,v \in L^2(Q_T)$, additionally for each $ t \in (0,T)$, $u(t), v(t) \in H^1_0(\Omega)$,
we use the standard definition of Dirichlet form
\begin{equation*}
  D(u,v) =  \langle \nabla u, \nabla v \rangle_{L^2(Q_T)} .
\end{equation*}

Further, on $Q_T$, we introduce the following Bochner spaces:
\begin{equation*}
	\begin{aligned}
	L^2&:= L^2(Q_T)  =L^2(0,T;L^2(\Omega)) \quad \mbox{with norm} \quad
	\|v \|_{L^2(Q_T)}^2 = \int_{Q_T} v^2 dx dt,    \\
	V &:= V(Q_T) = L^2(0,T; H^1_0(\Omega))  \quad \mbox{with norm} \quad
	\| v \|_V^2=
	        D(v,v) = \int_0^{T} \int_\Omega |\nabla v (t)|^2 dx dt,\\
	V^*&:= V(Q_T)^* = L^2(0,T; H^{-1}(\Omega))\quad \mbox{with norm}  ~~
	\|v \|_{V^*}= \sup_{ \phi \in V} \frac{ \langle v, \phi \rangle_{L^2(Q_T)}}{\|\phi\|_{V}}.
	\end{aligned}
\end{equation*}
We can use an equivalent shorthand notation $L^2(0,T;X(\Omega))$ for these norms
\[
\| v \|_{L^2(0,T;X(\Omega ))}^2: =  \int_0^T  \| v(t, \cdot) \|^2_{X(\Omega )}  dt.
\]
Below we will also use $\langle\cdot,\cdot\rangle_{L^2(Q_T)}$ for the duality pairing between $V$ and $V^*$.
For any $0<s<1$, we define the function space $B^s(Q_T)$ by
\begin{equation*}
   \Bs(Q_T)=\widetilde{H}_L^s(0,T; H^{-1}(\Omega)) \cap L^2(0,T;H_0^{1}(\Omega)).
\end{equation*}
The space is endowed with the following norm
\begin{equation}\label{BL-norm}
\| v \|^2_{\Bs(Q_T)} = \| {_0\partial_t^s} v \|^2_{V^*} + D(v,v).
\end{equation}
\begin{lemma}\label{lem:Bs-norm}
For any  $v \in B^\alpha(Q_T)$ with $\alpha\in(0,1)$, there holds $\| v \|_{\widetilde{H}_L^\alpha(0,T; H^{-1}(\Omega))} \sim \| {_0\partial_t^\alpha} u \|_{V^*}$.
\end{lemma}
\begin{proof}
By either \cite[Theorem 3.1]{GorenfloLuchkoYamamoto:2015} or \cite[Theorem 3.1]{Jin_MathComp_2015variational},
the norm equivalence $\|v(t, \cdot )\|_{\widetilde H_L^\alpha (0,T)}
\sim\| {{}_0\partial _t^\alpha v(t, \cdot )}\|_{{L^2}(0,T)}$ holds. Then the desired assertion follows
from the definition of the norms.
\end{proof}

The next two results give non-negativity of the fractional
integral and derivative operators.
\begin{lemma}\label{lem:pdf}
For any $v \in V$, we have $D( {_0I_t^{\alpha}} v ,  v ) \ge 0$.
\end{lemma}
\begin{proof}
Let $\widetilde v$ be the extension of $v$ to $\Omega \times \mathbb{R}$ by zero. Then clearly, we have
\begin{equation*}
  {_0 I_t^{\alpha}} v(t) = \frac{1}{\Gamma(\alpha)}\int_{-\infty}^t(t-s)^{\alpha-1}\widetilde v(s)ds := {_{-\infty} I_t^{\alpha}}\widetilde v (t).
\end{equation*}
With $~~\widehat{}~~$ being the Fourier transform in time, by Parseval's identity, we have
\begin{equation*}
\begin{split}
D( {_0 I_t^{\alpha}} v ,  v ) &= \int_0^T\int_\Omega(\nabla{_0 I_t^{\alpha}}v)\cdot(\nabla v )\,dx\,dt
= \int_0^T\int_\Omega ({_0 I_t^{\alpha}}\nabla v)\cdot(\nabla v )\,dx\,dt\\
&= \int_{-\infty}^\infty\int_\Omega ({_{-\infty} I_t^{\alpha}}\nabla \widetilde v)\cdot(\nabla \widetilde v )\,dx\,dt
=\int_\Omega\int_{-\infty}^\infty ({_{-\infty} I_t^{\alpha}}\nabla \widetilde v)\cdot(\nabla \widetilde v )\,dt\,dx\\
&=\int_\Omega\int_{-\infty}^\infty \widehat { {_{-\infty} I_t^{\alpha}}\nabla\widetilde v }(\xi)\cdot \widehat { \nabla\widetilde v }(\xi)\,dt\,dx= c_\alpha \int_\Omega\int_{0}^\infty |\xi|^{-\alpha} |\nabla \widetilde v|^2 \,d\xi\,dx  \ge 0,
\end{split}
\end{equation*}
where the last identity follows from $\widehat{_{-\infty}I_t^\alpha f}(\xi)=(-\mathrm{i}\xi)^{-\alpha}\widehat{f}(\xi)$ \cite[pp. 90]{KilbasSrivastavaTrujillo:2006}.
\end{proof}

\begin{lemma}\label{lem:pdf2}
For $v \in B^\alpha(Q_T)$, then $\langle{_0\partial_t^\alpha} v, v \rangle_{L^2(Q_T)}\ge 0$.
\end{lemma}
\begin{proof}
For $v(\cdot,t)\in {{\widetilde H}^\alpha _L}( {0,T})$, let $v_\alpha = {_0\partial_t^\alpha}v $.
Since $_0I_t^\alpha $ is the left inverse of the operator $_0\partial_t^\alpha $ on the space $\widetilde{H}_L^\alpha(0,T)$
\cite[pp. 75, Lemma 2.6]{KilbasSrivastavaTrujillo:2006}, we have $v={}_0I_t^\alpha v_\alpha $, and
\begin{equation*}
\langle {_0\partial _t^\alpha v,v} \rangle_{L^2(Q_T)}  =\langle
{{v_\alpha },{}_0I_t^\alpha {v_\alpha }}\rangle_{L^2(Q_T)} ,
\end{equation*}
which together with Parseval's identity concludes the proof.
\end{proof}
	
\subsection{Weak time-space formulation}
	
Inspired by the recent works \cite{mollet2014stabilityNo,Steinbach_CMAM_2015} on space-time
formulations for standard parabolic problems, we develop such
formulation for problem \eqref{orig} as well.
First we introduce the bilinear form $a(\cdot,\cdot): ~~\Bal(Q_T) \times V(Q_T) \to \mathbb{R}$:
\begin{equation}\label{eq:forma}
a(v, \phi) := \langle {_0\partial_t^\alpha} v, \phi \rangle_{L^2(Q_T)} + D(v, \phi).
\end{equation}
Then the  weak (Petrov-Galerkin) form of problem \eqref{orig}  is:
find $ u \in \Bal(Q_T)$  such that
\begin{equation}\label{eqn:BV-weak}
 	a(u, \phi) = \langle f, \phi \rangle_{L^2(Q_T)} \quad \forall \phi \in V.
\end{equation}

By Lemma \ref{lem:Bs-norm}, the bilinear form $a(\cdot,\cdot)$ is continuous on
the product space $B^\alpha(Q_T) \times V(Q_T)$:
\begin{equation*}
	| a(v, \phi) | \le |\langle {_0\partial_t^\alpha} v, \phi \rangle_{L^2(Q_T)} | + |D(v, \phi)| \le  \| v \|_{\Bal(Q_T)} \|\phi \|_V.
\end{equation*}
	
Next we show the inf-sup condition of the bilinear form $a(\cdot,\cdot)$.
\begin{lemma}[inf-sup condition]\label{lem:inf-sup}
For all $ v \in \Bal(Q_T)$, there holds
\begin{equation}\label{eq:inf-sup}
	\sup_{\phi \in V} \frac{a(v,\phi) }{\|\phi \|_V} \ge \|v \|_{\Bal(Q_T)}.
\end{equation}
Moreover, for any $ \phi \in V$, with $\phi\neq0$,  the following compatibility condition holds:
\begin{equation*}
   \sup_{v \in \Bal(Q_T)} a(v,\phi) >0.
\end{equation*}
\end{lemma}
\begin{proof}
First, following \cite{Steinbach_CMAM_2015} we introduce
a {Newton potential} operator $N : V^* \to V$ as $ N \, \psi = w$,
where $w$ is a solution of the following problem: find $w \in V$ such that
\begin{equation*}
	D(w, \phi) =\langle  \psi, \phi\rangle_{L^2(Q_T)} \quad \forall \phi \in V.
\end{equation*}
By Lax-Milgram theorem, this problem has a unique solution $w = N \psi \in V$ that satisfies
\begin{equation}\label{eqn:pot}
   \| w \|_V = \|N \psi \|_V = \| \psi\|_{V^*}.
\end{equation}
For any given $ v \in \Bal(Q_T)$, let $ \phi_v  = v + N  _0\partial_t^\alpha v$. Obviously,
$ \phi_v \in V$ and by \eqref{eqn:pot}
\begin{equation*}
	\| \phi_v \|_V = \| v + N  _0\partial_t^\alpha v \|_V \le \|v \|_V + \|N _0\partial_t^\alpha v \|_V
	=  \|v \|_V + \|_0\partial_t^\alpha v \|_{V^*} = \| v \|_{\Bal(Q_T)}.
\end{equation*}
Using the function $\phi_v$, we have
\begin{align*}
a(v,\phi_v) & = \langle {_0\partial_t^\alpha } v, \phi_v \rangle_{L^2(Q_T)} + D(v, \phi_v) \\
& = \langle {_0\partial_t^\alpha} v, v \rangle_{L^2(Q_T)} + D(v, v) + \langle {_0\partial_t^\alpha} v, N {_0\partial_t^\alpha}v \rangle_{L^2(Q_T)} + D(v, N {_0\partial_t^\alpha} v).
\end{align*}
By the definition of the operator $N$, we have
\begin{equation*}
  \langle {_0\partial_t^\alpha} v, N {_0\partial_t^\alpha} v \rangle_{L^2(Q_T)} = D( N{_0\partial_t^\alpha} v ,N{_0\partial_t^\alpha}v)\quad\mbox{and}\quad
D( v, N {_0\partial_t^\alpha} v ) = \langle {_0\partial_t^\alpha} v, v \rangle_{L^2(Q_T)},
\end{equation*}
and consequently,
\begin{equation*}
a(v,\phi_v) = 2 \langle {_0\partial_t^\alpha}  v, v \rangle_{L^2(Q_T)} + D(v, v) + D( N{_0\partial_t^\alpha}  v, N {_0\partial_t^\alpha} v).
\end{equation*}
Then Lemmas \ref{lem:Bs-norm} and \ref{lem:pdf2} and \eqref{eqn:pot} yield
\begin{equation*}
a(v,\phi_v) \ge \| v \|_{V}^2 + \| _0\partial_t^\alpha v  \|_{V^*}^2 = \| v \|_{\Bal(Q_T)}^2.
\end{equation*}
This completes the proof of the {inf-sup} condition.
	
Next, we prove the compatibility condition.  For a given
$0\neq \phi \in V$, let $v_\phi= {_0I_t^{\alpha}} \phi $. Then
\begin{equation*}
  _0\partial_t^\alpha  v_\phi  = {_0\partial_t^\alpha} ( {_0I_t^{\alpha}} \phi)= \phi.
\end{equation*}
Thus, $\| {_0\partial_t^\alpha} v_\phi \|_{L^2(Q_T)}  =  \| \phi\|_{L^2(Q_T)} \le
\| \phi\|_{V}$ and as a result, $ v_\phi \in \Bal(Q_T)$
and
$$
\langle {_0\partial_t^\alpha} v_\phi, \phi \rangle_{L^2(Q_T)} = \langle \phi, \phi \rangle_{L^2(Q_T)} =\|\phi\|_{L^2(Q_T)}^2>0.
$$
The required bound
$ \sup_{v \in \Bal} a(v,\phi) >0 $  follows easily from the inequality
$D( {_0I_t^{\alpha}} \phi ,  \phi)  \ge 0$ (cf. Lemma \ref{lem:pdf}), which completes the proof of the lemma.
\end{proof}

Now we show  the existence and uniqueness of the weak solution.
\begin{theorem}\label{thm:existence}
For any $f\in V^*(Q_T)$, problem \eqref{eqn:BV-weak} has a unique solution
$ u \in \Bal(Q_T)$, and it satisfies the following a priori estimate
$$
	\| u \|_{B^\alpha(Q_T)} \le c \|f \|_{V^*(Q_T)}.
$$
\end{theorem}
\begin{proof}
The existence, uniqueness and stability follow immediately from Lemma \ref{lem:inf-sup}, and the continuity
of the bilinear form $a(\cdot,\cdot)$.
\end{proof}
	
\begin{remark}
Li and Xu \cite{Li_Xu_SINUM_2009} proposed the following Galerkin weak formulation:
find  $ u \in \tilde B^\frac{\alpha}{2}(Q_T): = H^\frac{\alpha}{2}(0,T;L^2(\Omega))\cap L^2(0,T;H_0^1(\Omega))$ such that
$a(u,v) = \langle f, v \rangle_{L^2(Q_T)}, \  \forall v \in \tilde B^\frac{\alpha}{2}(Q_T),$
with the bilinear form $a(\cdot,\cdot)$ defined by
\begin{equation*}
	a(u,v) = \langle _0\partial_t^\frac{\alpha}{2} u,\ {_t\partial_T^\frac{\alpha}{2}v} \rangle_{L^2(Q_T)} + D(u,v).
\end{equation*}
The bilinear form $a(\cdot, \cdot)$ is continuous and coercive on the space $ \tilde B^\frac{\alpha}{2}(Q_T)$, cf.
\cite{Li_Xu_SINUM_2009}, and thus the variational problem is well posed.
Further, they studied a  spectral approximation. For other interesting
extensions of space-time fractional models, one can find in \cite{LiXu:2010,ZayernouriAinsworth:2015}.
\end{remark}

\begin{remark}
Note that for $f\in L^2(0,T;L^2(\Omega))$, the initial condition $u(0)=0$ in \eqref{orig}
makes sense only if $\alpha>\frac12$. For $\alpha\leq\frac12$, one should
not impose any initial condition, unless $f$ has extra temporal regularity, which however may be
interpreted in a weak sense  \cite{GorenfloLuchkoYamamoto:2015}.
\end{remark}

\subsection{Regularity of the solution}
If the source term $f$ has higher  spatial and/or temporal  regularity, then accordingly,
the solution $u$ is more regular than that in Theorem \ref{thm:existence}.
Now we establish such regularity pickup, which is useful for
the error analysis in Section \ref{sec:conv-PDE}.

Let $\{\varphi_n\}_{n=1}^{\infty}\in H^2(\Omega) \cap H^1_0(\Omega)$ and $\{\lambda_n\}_{n=1}^\infty$
denote respectively the $L^2(\Omega)$-orthonormal eigenfunctions of the operator $-\Delta$ (with
a homogeneous Dirichlet boundary condition) and corresponding eigenvalues
(ordered non-decreasingly with multiplicity counted).
Then the solution $u$ of problem \eqref{orig} can be expressed by
\begin{equation}\label{eqn:sol-rep}
  u(t) = \int_0^t E(t-s)f(s) ds = \int_0^t E(s)f(t-s)ds.
\end{equation}
Here the solution operator $E(t)$ is defined by
$ E(t)v = t^{\alpha-1}E_{\alpha,\alpha}( t^\alpha\Delta)v,$
where for any $\alpha>0$ and $\beta\in\mathbb{R}$,
$E_{\alpha,\beta}(z)=\sum_{k=0}^\infty z^k/\Gamma(k\alpha+\beta)$ \cite[pp. 42]{KilbasSrivastavaTrujillo:2006}.

The next result gives the solution stability and regularity pick up.
\begin{theorem}\label{thm:regularity}
For $f\in {{\widetilde H}_L^s}( {0,T;{L^2}(\Omega )})$,  $s\in [0,1]$, the solution $u$
to problem \eqref{orig} belongs to $\widetilde H_L^{\alpha  + s}( {0,T;{L^2}(\Omega )}) \cap
\widetilde H_L^{ s}(0,T;H_0^1(\Omega) \cap H^2(\Omega))$, and
\begin{equation}\label{Reg-est-u}
  \|u\|_{\widetilde{H}^{\alpha+s}_L (0,T;L^2(\Omega))} + \|u\|_{\widetilde{H}^s_L (0,T;H^2(\Omega))}
     \leq c\|f\|_{\widetilde{H}^s_L (0,T;L^2(\Omega))}.
\end{equation}
Furthermore, if $f\in {{\widetilde H}_L^s}( {0,T;{H_0^1}(\Omega )})$ then
$u\in \widetilde H_L^{\alpha + s}( {0,T;{H_0^1(\Omega)}})$.
\end{theorem}
\begin{proof}
For $f\in L^2(Q_T)$, by \cite[Theorem 4.1]{GorenfloLuchkoYamamoto:2015}, there holds
\begin{equation*}
 \|{}_0\partial_t^\alpha u\|_{L^2(Q_T)} + \|\Delta u\|_{L^2(Q_T)}\leq c\|f\|_{L^2(Q_T)},
\end{equation*}
which shows the assertion for $s=0$. Next we turn to the case $s=1$.
Since $f\in \widetilde {H}^1_L(0,T;L^2(\Omega))$,
we have $f(0)=0$. Now by differentiating the
representation \eqref{eqn:sol-rep} with respect to $t$, we deduce
\begin{equation*}
  u^\prime(t) = E(t)f(0) + \int_0^tE(t-s)f^\prime(s)ds = \int_0^tE(t-s)f^\prime(s)ds.
\end{equation*}
By the preceding estimate, the term $v=:\int_0^tE(t-s)f^\prime(s)ds$ satisfies
\begin{equation*}
  \| {_0\partial_t^\alpha } v\|_{L^2(Q_T)} + \|\Delta v\|_{L^2(Q_T)} \leq c\|f^\prime\|_{L^2(Q_T)}.
\end{equation*}
By \cite[Lemma 2.2]{JinLazarovPasciakZhou:2015} and the fact that
$L^\infty( 0,T ; L^2(\Omega))\subset{{\widetilde H}_L^1}( {0,T;L^2(\Omega )})$,  for $s=1$ we get
\begin{equation*}
     \| u(t) \|_{L^2(\Omega)} \le c \int_0^t (t-s)^{\alpha-1} \| f (s)\|_{L^2(\Omega)}\,ds
     \le ct^{\alpha} \| f \|_{L^\infty( 0,T ; L^2(\Omega))}  \rightarrow 0 \quad \text{as}~~t\rightarrow0,
\end{equation*}
and hence $u(0)=0$. Consequently, we have
$$ {}_0\partial_t^\alpha v(t) = \partial_t ({_0\hspace{-0.3mm}I^{1-\alpha}_t} (\partial_t u))(t)
= \partial_t^2({_0\hspace{-0.3mm}I^{1-\alpha}_t}  u )(t)
= {}_0\partial_t^{\alpha+1}u(t).$$
Therefore, the desired assertion holds also for $s=1$. Since only temporal regularity index is concerned,
one may apply an interpolation argument to deduce the intermediate case (see \cite[Lemma 2.8]{MeyriesSchnaubelt:2012}
or \cite[Theorem 2.35]{HytonenWeis:2016}). The case $f\in \widetilde H^s_L(0,T;H_0^1(\Omega))$ follows similarly.
\end{proof}

\section{Petrov-Galerkin FEM  for tensor-product meshes}\label{sec:FEM}
Based on the space-time variational formulation in Section \ref{sec:space-time},  in this part,
we develop a novel Petrov-Galerkin finite element method (FEM), establish the discrete inf-sup condition
and describe its linear algebraic formulation.

\subsection{Finite element method}
First, we introduce a quasi-uniform shape regular partition of the domain $\Omega$ into simplicial elements
of maximal diameter $h$, which is denoted by ${\mathcal T}_h$. We consider the space of continuous
piecewise linear functions on ${\mathcal T}_h$ with $N\in\mathbb{N}$
being the number of degrees of freedom. Let $\{\varphi_i\}_{i=1}^N \subset H^1_0(\Omega)$
be the nodal basis functions and denote
\begin{equation*}
\mathbb{X}_h := \text{span} ( \{\varphi_i\}_{i=1}^N).
\end{equation*}
On the space $\mathbb{X}_h$, we recall the $L^2$-projection $P_h:L^2(\Omega)\to \mathbb{X}_h$ defined by:
\begin{equation}\label{L2-P}
   (\phi-P_h\phi, \chi)_{L^2(\Omega)} = 0 \quad \forall \chi\in \mathbb{X}_h.
\end{equation}
It is well known that it satisfies the following error estimate \cite{Thomee:2006}:
\begin{equation}\label{eqn:err-Ph}
  \|P_h\phi-\phi\|_{L^2(\Omega)} + h\|P_h\phi-\phi\|_{H^1(\Omega)} \leq ch^q\|\phi\|_{H^q(\Omega)},
             \quad \forall \phi\in H_0^1(\Omega)\cap H^q(\Omega),\;  q=1,2,
\end{equation}
and the following negative norm estimate \cite[pp. 69]{Thomee:2006}:
\begin{equation}\label{neg-norm-Ph}
\|P_h\phi-\phi\|_{H^{-1}(\Omega)}\le c h\|\phi\|_{L^2(\Omega)}.
\end{equation}

Next we uniformly partition the time interval $(0,T)$ with grid points $t_k=k \tau$, $ k=0,\ldots,K$,
$K\in\mathbb{N}$, and a time step size $\tau=T/K $. On this partition, following \cite{Jin2015petrovSINUM},
we define a set of ``fractionalized'' piecewise constant basis functions $ \phi_k(t)$,   $k =1,\ldots,K$, by
\begin{equation}\label{eqn:fy}
\phi_k(t) = \left\{
\begin{array}{ll}
 0,           &     0 \le t  \le t_{k-1}, \\
(t - t_{k-1})^{\alpha}, &  t_{k-1}  \le  t  \le T,
\end{array} \right \}
:=(t - t_{k-1})^{\alpha}\chi_{[t_{k-1}, T]}(t),
\end{equation}
where $\chi_S$ denotes the characteristic function of the set $S$. It is
easy to verify that for $k=1, \dots, K$
$$
\phi_k(t)=\Gamma(\alpha +1){_0I_t^{\alpha}}\chi_{[t_{k-1}, T]}(t)
\quad \mbox{and} \quad
_0\partial_t^\alpha \phi_k (t) =\Gamma(\alpha +1) \chi_{[t_{k-1}, T]}(t).
$$
Clearly, $\phi_k \in \widetilde H_L^{\alpha+s}(0,T)$ for any $s\in[0,1/2)$.

Further, we introduce the following spaces
\begin{equation}\label{eqn:spaceVt}
\bVt= \text{span}  \{\phi_k(t) \}_{k=1}^K
\quad\mbox{and}\quad \mathbb{W}_\tau := \text{span} \{ \chi_{[t_{k-1}, T]}(t)  \}_{k=1}^K.
\end{equation}
With the tensor product notation $\otimes$, the solution space $\bBal \subset \Bal(Q_T)$
and  test space $ \bVht  \subset  V(Q_T)$ are respectively defined by
\begin{equation}\label{eqn:BV}
\bBal :=  \mathbb{X}_h \otimes \bVt \quad\mbox{and}\quad
\bVht  := \mathbb{X}_h \otimes \mathbb{W}_\tau.
\end{equation}
The functions in the spaces $\bBal$ and $\bVht$ are products
of standard $C^0$-conforming finite element in space and
``fractionalized'' piecewise constant and piecewise constant functions in time, respectively.
The FEM problem of \eqref{eqn:BV-weak} reads: given $f\in V^*$,
find $u_{h\tau} \in \bBal$ such that
\begin{equation}\label{eqn:BV-weak-h}
 a(u_{h \tau}, \phi)  \equiv
\langle {_0\partial_t^\alpha} u_{h \tau}, \phi \rangle _{L^2(Q_T)} + D(u_{h \tau}, \phi)
                           = \langle f, \phi \rangle_{L^2(Q_T)} \quad \forall \phi \in \bVht.
\end{equation}
The bilinear form $ a(\cdot, \cdot)$ is non-symmetric, and to show the existence and stability of the
solution  $u_{h \tau}$, we need to establish a discrete analogue of the inf-sup condition \eqref{eq:inf-sup}. To
prove this, we first introduce and study the $L^2$-projection onto the space $\mathbb{W}_\tau$.

\subsection{The projection $\Pi_\tau$ and its properties}
\label{sec:L2projection}

For functions $v(t)$ defined on $(0,T)$, we introduce the $L^2$-projection
$\Pi_\tau:~L^2(0,T)\to \mathbb{W}_\tau$ by
\begin{equation}\label{eqn:L2-orth-W1}
(\Pi_\tau v, \phi)_{L^2(0,T)}= (v, \phi)_{L^2(0,T)},  \quad \forall \phi \in \bWt,
\end{equation}
where  $(\cdot,\cdot)_{L^2(0,T)}$ denotes the inner product on the space $L^2(0,T)$,
or equivalently
$$
(\Pi_\tau v) (t) = \tau^{-1}\int_{t_n}^{t_{n+1}} v(t) \, dt, \quad \text{for} \quad t\in [t_n,t_{n+1}) :=[t_n, t_n+\tau).
$$
Then the operator $\Pi_\tau$ satisfies the error estimate:
\begin{equation}\label{eqn:err-L2-W}
\|v-\Pi_\tau v\|_{L^2(0,T)}\leq c\tau^s\|v\|_{H^s(0,T)},\quad s\in[0,1].
\end{equation}
Below, we study the $L^2$-stability
of the operator $\Pi_\tau$  when restricted to the space $\bVt$.  This is
given in Lemma~\ref{ptau-stability} below, whose proof will require the next result.

\begin{lemma}\label{move-alpha}
For $u\in H^\frac{\alpha}{2}(0,T)$ and $v\in\widetilde H_L^\alpha(0,T)$,
\begin{equation*}
  (u,{_0\partial_t^\alpha v})_{L^2(0,T)} =({_t\partial_T^\frac{\alpha}{2}} u,\,{_0\partial_t^\frac{\alpha}{2}} v)_{L^2(0,T)}.
\end{equation*}
\end{lemma}
\begin{proof}
Let $u$ and $v$ be as in the lemma.  Let $\widetilde u $ denote the extension of $u$ by zero to $\mathbb{R}$ and
\begin{equation*}
\widetilde v(t) =\left \{ \begin{aligned} 0:&\qquad \hbox{if }t\notin
	[0,2T],\\
	v(t):& \qquad \hbox{if }t\in[0,T], \\
	v(2T-t):& \qquad \hbox{if } t\in (T,2T].
	\end{aligned} \right .
\end{equation*}
Then there holds
\begin{equation*}
(u,{_0\partial_t^\alpha v})_{L^2(0,T)} = (\widetilde u,{\,_{-\infty}\partial_t^{\alpha}} \widetilde v)_{L^2(\mathbb{R})} =
	({\,_t\partial_\infty^\frac{\alpha}{2}} \widetilde u,{\,_{-\infty}\partial_t^\frac{\alpha}{2}} \widetilde v)_{L^2(\mathbb{R})}
	=({_t\partial_T^\frac{\alpha}{2}} u,\,{_0\partial_t^\frac{\alpha}{2}}v)_{L^2(0,T)},
\end{equation*}
where the middle equality followed by examining the expressions after
applying the Fourier transform as in the proof of Lemma~\ref{lem:pdf}.
\end{proof}
By using arguments similar to Lemma \ref{lem:pdf2}
(cf. also \cite[pp. 90]{KilbasSrivastavaTrujillo:2006})
we conclude that that there is a constant $c_\alpha$ satisfying
\begin{equation}\label{u-dal}
c_\alpha^{-1} \|v\|^2_{H^\frac{\alpha}{2}(0,T)} \le (v,{\,_0\partial_t^\alpha}v)_{L^2(0,T)} \le
c_\alpha\|v\|^2_{H^\frac{\alpha}{2}(0,T)},\quad\forall v\in \widetilde H_L^\alpha(0,T).
\end{equation}

\begin{lemma}\label{ptau-stability}
There is a constant $c(\alpha)>0$ such that
\begin{equation*}
c(\alpha) \|v\|^2_{L^2(0,T)} \le \|\Pi_\tau v \|^2_{L^2(0,T)} \le \|v\|^2_{L^2(0,T)}, \quad\forall
	v\in \bVt.
\end{equation*}
\end{lemma}
\begin{proof}
The second inequality follows directly from the definition of $\Pi_\tau$. For the first, we note
\begin{equation}\label{ortho-v}
	\|v\|_{L^2(0,T)}^2= \|\Pi_\tau v\|_{L^2(0,T)}^2+ \|(I-\Pi_\tau)v\|_{L^2(0,T)}^2.
\end{equation}
The approximation property \eqref{eqn:err-L2-W} implies that for any $v \in \bVt$
\begin{equation}
	\|(I-\Pi_\tau)v\|_{L^2(0,T)}\le c \tau^\frac{\alpha}{2} \|v\|_{H^\frac{\alpha}{2}(0,T)}.
	\label{impt}
\end{equation}
Since for any $v\in\bVt$, ${\,_0\partial_t^\alpha} v$ belongs to $\bWt$, Lemma~\ref{move-alpha} and \eqref{u-dal} imply
\begin{equation}
	\begin{aligned} c^{-1}  \|v\|_{ H^\frac{\alpha}{2}(0,T)}^2 &\le (v,{\,_0\partial_t^\alpha} v)_{L^2(0,T)}
	=(\Pi_\tau v,{_0\partial_t^\alpha} v)_{L^2(0,T)}\\
	&=({\,_t\partial_T^\frac{\alpha}{2}} (\Pi_\tau v),{\,_0\partial_t^\frac{\alpha}{2}} v)_{L^2(0,T)} \le c \|\Pi_\tau v
	\|_{H^\frac{\alpha}{2}(0,T)} \|v\|_{H^\frac{\alpha}{2}(0,T)}.
	\end{aligned}
	\label{partI}
\end{equation}
Note that  $\bWt$ satisfies inverse inequalities,  namely, for $s\in (0,1/2)$,
\begin{equation*}
\|\Pi_\tau v\|_{H^{s}(0,T)} \le c_s \tau^{-s} \| \Pi_\tau v \|_{L^2(0,T)}.
\end{equation*}
Using this with $s=\alpha/2$ in \eqref{partI} implies
\begin{equation}	\label{hnal2}
	\| v\|_{H^\frac{\alpha}{2}(0,T)} \le c \tau^{-\frac{\alpha}{2}} \| \Pi_\tau v\|_{L^2(0,T)}.
\end{equation}
Substituting \eqref{hnal2} into \eqref{impt} and combining \eqref{ortho-v} give
\begin{equation*}
   \|v\|_{L^2(0,T)} \le c \|\Pi_\tau v\|_{L^2(0,T)},
\end{equation*}
which completes the proof of the lemma.
\end{proof}

In Table \ref{tab:Pi-bound}, we give the
best constant  $c(\alpha)\equiv c(\alpha,K)$ (recall $\tau K =T$) as a function of the mesh parameter $K=T/\tau$
when $T=1$.  The results clearly show the convergence to a lower bound as
$K$ becomes large.  Further, we note that it is a consequence of  the work of
Larsson and Monteli \cite{LarssonMolteni:2016}, that $c(\alpha) \to 0$ when $\alpha \to 1$,
for which our discretization coincides with that in \cite{LarssonMolteni:2016}.

\begin{table}[hbt!]
	\caption{The lower bound $c(\alpha)$ for the $L^2$-norm of $\Pi_\tau $ for various $\alpha$.}
	\label{tab:Pi-bound}
	\centering
	\begin{tabular}{|c|cccccc|}
		\hline
		$\alpha {\backslash}  K$   & $20$ &$40$ & $80$ & $160$ & $320$ & $640$ \\
		\hline
		0.3     & 0.7711 & 0.7697 & 0.7693 & 0.7693 & 0.7693 & 0.7692 \\
		\hline
		0.5      & 0.4754 & 0.4714  & 0.4703 & 0.4700 & 0.4700   & 0.4699\\
		\hline
		0.7      & 0.1982 & 0.1911 &  0.1891 & 0.1886 & 0.1884  &  0.1884 \\
		\hline
		0.9      & 0.0326 & 0.0251 &  0.0228 & 0.0221 & 0.0220 & 0.0219 \\
		\hline
             0.98    & 0.0076 & 0.0030 &  0.0015 & 0.0011 & 0.0010 & 0.0010      \\
             \hline
	\end{tabular}
\end{table}

\subsection{Stability of the Petrov-Galerkin FEM}\label{sec:FEM_stability}
Now we prove a discrete inf-sup condition to guarantee existence and uniqueness of the
finite element solution $u_{h\tau}$. In this part, we consider $K$, $N$, $\tau$ and $h$ as
fixed, although the estimates are independent of them.

Let $\{\psi_{j}\}_{j=1}^N\subset \mathbb{X}_h$ denote an $L^2(\Omega)$-orthonormal basis for $\mathbb{X}_h$
of generalized eigenfunctions (of the negative discrete Laplacian), i.e.,
\begin{equation*}
  (\nabla \psi_j,\nabla \chi)_{L^2(\Omega)} = \lambda_{j,h} (\psi_j,\chi)_{L^2(\Omega)} \quad \forall \chi\in \mathbb X_h.
\end{equation*}
It follows that for any $\phi\in \mathbb X_h$, there hold
\begin{equation*}
\phi = \sum_{j=1}^N (\phi,\psi_j)_{L^2(\Omega)} \psi_j,\quad \|\phi\|^2_{L^2(\Omega)}=\sum_{j=1}^N  (\phi,\psi_j)_{L^2(\Omega)}^2
\quad \hbox{and} \quad
\|\nabla \phi\|^2_{L^2(\Omega)}=\sum_{j=1}^N \lambda_{j,h}  (\phi,\psi_j)_{L^2(\Omega)}^2.
\end{equation*}
We also define
\begin{equation*}
N_h \phi = \sum_{j=1}^N \lambda_{j,h}^{-1} (\phi,\psi_j)_{L^2(\Omega)}
\psi_j\quad \hbox{and} \quad
\|\phi\|^2_{H^{-1}_h(\Omega)}=\sum_{j=1}^N \lambda_{j,h}^{-1}
(\phi,\psi_j)_{L^2(\Omega)}^2.
\end{equation*}
The operator $N_h$ is a discrete Riesz map, i.e., the inverse of the discrete Laplacian on
the space $\mathbb{X}_h$. It is well known that there is a constant $c$ independent of $h$ satisfying
\begin{equation}\label{hm1hm1h}
\|\phi \|_{H^{-1}(\Omega)}\le c \|\phi\|_{H^{-1}_h(\Omega)} \quad \forall \phi\in \mathbb X_h.
\end{equation}
Further, due to the tensor construction of the spaces $\bBal$ and $\bVht$,
functions $v\in\bBal$ and $\phi\in \bVht$ can be expanded as
\begin{equation*} 
{v(x,t)}=\sum_{i,j} c_{ij}\, \phi_i(t) \psi_j(x)
\quad \mbox{and} \quad {\phi(x,t)}=\sum_{i,j} d_{ij} \chi_i(t) \psi_j(x),
\end{equation*}
were the summation over $i,j$ denotes the sum over $i=1,\ldots , K$ and $j=1,\ldots N$.
This discussion extends to  $Q_T$ as well: for example, for $v\in \bBal$, we have
the following expansion (with $
v_j(t) =  (v(\cdot,t),\psi_j)_{L^2(\Omega)}$)
\begin{equation*}
  \begin{aligned}
  D(v,v) &= \sum_{j=1}^N \lambda_{j,h}   \|v_j(t)\|_{L^2(0,T)}^2, \\ 
 \|{\,_0\partial_t^\alpha} v \|_{L^2(Q_T)}^2 & = \sum_{j=1}^N \|{\,_0\partial_t^\alpha} v_j(t)\|_{L^2(0,T)}^2 ,\\
  \|{\,_0\partial_t^\alpha} v \|^2_{L^2(0,T;H^{-1}(\Omega))} &\le c \|{\,_0\partial_t^\alpha} v \|_{L^2(0,T;H_h^{-1}(\Omega))}^2=c \sum_{j=1}^N \lambda_{j,h}^{-1}
  \|{\,_0\partial_t^\alpha} v_j(t)\|_{L^2(0,T)}^2 \, .
  \end{aligned}
\end{equation*}
Similarly, for $\phi\in \bVht$, we have
\begin{equation*}
\|\phi\|_V^2=
\sum_{j=1}^N \lambda_{j,h}\|\phi_j(t)\|_{L^2(0,T)}^2 \,\quad \mbox{with } \phi_j(t)=(\phi(\cdot,t),\psi_j)_{L^2(\Omega)}.
\end{equation*}

Now we give a discrete inf-sup condition, which
implies the well-posedness of problem \eqref{eqn:BV-weak-h}.

\begin{lemma}
\label{discr-inf-sup}
There is a constant $c_\alpha>0$, independent of $h$ and $\tau$, such that
\begin{equation}\label{dinf-sup}
 \sup_{\phi \in \bVht} \frac{a(v,\phi) }{\|\phi \|_V} \ge c_\alpha\|v\|_{B^\alpha(Q_T)}
               \quad \forall   v \in \bBal.
\end{equation}
\end{lemma}
\begin{proof}
For any $v\in\bBal$, we define a norm
\begin{equation*}
	\tribar v \tribar^2= \|{\,_0\partial_t^\alpha} v\|_{L^2(0,T;H^{-1}_h(\Omega))}^2 +D(\Pi_\tau
	v,\Pi_\tau v).
\end{equation*}
Meanwhile, we set $\phi\in \bVht$ by
\begin{equation*}
	\phi=\left \{\begin{aligned}  \hbox{Case 1: }& N_h {\,_0\partial_t^\alpha} v\quad \hbox{ if } \|{\,_0\partial_t^\alpha} v\|_{L^2(0,T;H^{-1}_h(\Omega))}^2\ge D(\Pi_\tau	v,\Pi_\tau v),\\
	\hbox{Case 2: }&\phi= \Pi_h v\quad\hbox{otherwise.}
	\end{aligned} \right .
\end{equation*}
For $\phi$  given by Case 1, we have
\begin{equation*}
 a(v,\phi) = \|{\,_0\partial_t^\alpha} v \|_{L^2(0,T;H^{-1}_h(\Omega))}^2 + (v,{\,_0\partial_t^\alpha} v)_{{L^2(Q_T)}}
	\ge \|{\,_0\partial_t^\alpha} v \|_{L^2(0,T;H^{-1}_h(\Omega))}^2 \ge \tfrac 12 \tribar v \tribar^2.
\end{equation*}
By the definition of the operator $N_h$,
\begin{equation*}
  \begin{aligned}
	\|\phi\|_{ V}^2 =D(\phi,\phi) & =D(N_h{\,_0\partial_t^\alpha} v,N_h{_0\partial_t^\alpha} v)  =({\,_0\partial_t^\alpha} v,N_h {\,_0\partial_t^\alpha} v)_{L^2(Q_T)}\\
     &= \|{\,_0\partial_t^\alpha} v\|_{L^2(0,T;H^{-1}_h(\Omega))}^2\le \tribar v\tribar^2.
  \end{aligned}
\end{equation*}
Alternatively, if $\phi$ is given by Case 2, since ${\,_0\partial_t^\alpha} v\in \bVht $ for $v\in \bBal$, we derive
\begin{equation*}
  \begin{aligned}
   a(v,\phi) &=({\,_0\partial_t^\alpha} v,\Pi_\tau v)_{L^2(Q_T)}+D(v,\Pi_\tau v) \\
   & = ({\,_0\partial_t^\alpha} v,v)_{L^2(Q_T)} + D(\Pi_\tau  v,\Pi_\tau v)\ge  D(\Pi_\tau  v,\Pi_\tau v)\ge \tfrac 12 \tribar v\tribar^2.
  \end{aligned}
\end{equation*}
Also, by Lemma \ref{ptau-stability}, we have
\begin{equation*}
 c_\alpha \|\phi\|_{V}^2 \le D(\Pi_\tau
	v,\Pi_\tau v) \le  \tribar v\tribar^2.
\end{equation*}
Thus we have for any $v\in\bBal$
\begin{equation*}
	\tfrac{1}{2}\tribar v \tribar \le \frac {a(v,\phi)} {\tribar v\tribar} \le c  \frac {a(v,\phi)}
	{\|\phi\|_{V}}\le c  \sup_{\phi\in \bVht } \frac {a(v,\phi)}
	{\|\phi\|_{V}} .
\end{equation*}
Then applying \eqref{hm1hm1h} yields
\begin{equation*}
\big( \|{\,_0\partial_t^\alpha} v \|_{L^2(0,T;H^{-1}(\Omega))}^2 +D(v,v) \big)^\frac12  \le c\tribar v \tribar, \quad\forall v\in \bBal,
\end{equation*}
from which the desired inf-sup condition \eqref{dinf-sup} follows.
\end{proof}

\subsection{Linear algebraic problem}

Now we discuss the solution of the resulting linear system.
Let $\chi_\ell(t)=\chi_{[t_{\ell-1},t_\ell]}(t)$ and $\phi_k(t)=(t -
t_{k-1})^{\alpha}\chi_{[t_{k-1}, T]}(t)$. Then we define two matrices by:
$$
M_\tau = \{(\phi_k,\chi_\ell)_{L^2(0,T)} \}_{k,\ell=1}^K
\quad  \mbox{and}   \quad M^\alpha_\tau = \{ ({_0\partial_t^\alpha} \phi_k, \chi_\ell )_{L^2(0,T)} \}_{k,\ell=1}^K.
$$
Since ${_0\partial_t^\alpha}\phi_k (t) =\Gamma(\alpha +1) \chi_{[t_{k-1}, T]}(t)$, we obtain
\[ M_\tau=
\frac{\tau^{\alpha +1}}{\alpha +1} \left [
\begin{matrix}
d_1   & 0    & 0  &  \dots & 0 \\
d_{2}  &  d_1 &  0 & \dots  & 0  \\
d_{3}  & d_{2} & d_1 & \dots & 0 \\
\vdots & \vdots &\vdots &  \ddots & \\
d_{K}& d_{K-1}& d_{K-2}  & \dots & d_1
\end{matrix}
\right ] \quad \mbox{and} \quad
M^\alpha_\tau=
{\tau}{\Gamma(\alpha +1}) \left [
\begin{matrix}
1   & 0    & 0  &  \dots & 0 \\
1  &  1 &  0 & \dots  & 0  \\
1  &  1 & 1 & \dots & 0 \\
\vdots & \vdots &\vdots &  \ddots & \\
1 &  1 &  1   & \dots & 1
\end{matrix}
\right ].
\]
where $d_k=k^{\alpha+1}-(k-1)^{\alpha+1}$, $k=1,2,\ldots,K$.
The matrix $M_\tau$ is the temporal stiffness matrix, and it is Toeplitz.
Similarly, $M^\alpha_\tau$ is the temporal mass matrix.

Likewise, we introduce the ``mass" and ``stiffness" matrices related to the spatial variable $x$:
$$
M_h = \{ (\varphi_i, \varphi_j)_{L^2(\Omega)} \}_{ i,j=1}^N \quad \text{ and }
\quad A_h =  \{ ( \nabla \varphi_i,  \nabla \varphi_j)_{L^2(\Omega)} \}_{ i,j=1}^N,
$$
where $\varphi_i(x)$, $ i=1, \dots, N$, are the nodal basis functions of the space $ \mathbb{X}_h$.

We denote by $U$ the coefficient vector in the representation of the
solution $u_{h,\tau} \in \bBal$,
and by $F$ the vector of the projection of the source $f$ onto $\bVht $.
Then,  problem \eqref{eqn:BV-weak-h}
can be written as an algebraic system
\begin{equation*}
	AU=F, \quad \text{with}
	\quad A = M^\alpha_\tau \otimes M_h + M_\tau \otimes A_h.
\end{equation*}
Due to the block triangular structure of the matrix $A$,
the solution process is essentially time stepping, i.e., solving first for the unknowns
at $t_1=\tau$, and then recursively for $t_k$, $k=2, \dots, K$.

Alternatively, one may take
$$\chi_\ell(t)=\chi_{[t_{\ell-1},t_\ell]}(t) \quad \mbox{ and} \quad
\phi_k(t)=(t - t_{k-1})^{\alpha}\chi_{[t_{k-1}, T]}(t)-(t - t_{k})^{\alpha}\chi_{[t_{k}, T]}(t).
$$
Then with ${_0\partial_t^\alpha}\phi_k (t) =\Gamma(\alpha +1) \chi_{[t_{k-1}, t_k]}(t)$, the matrices $
M_\tau$ and $M^\alpha_\tau$ are given by
\[
M_\tau=
\frac{\tau^{\alpha +1}}{\alpha +1} \left [
\begin{matrix}
e_1   & 0    & 0  &  \dots & 0 \\
e_{2}  &  e_1 &  0 & \dots  & 0  \\
e_{3}  & e_{2} & e_1 & \dots & 0 \\
\vdots & \vdots &\vdots &  \ddots & \\
e_{K}& e_{K-1}& e_{K-2}  & \dots & e_1
\end{matrix}
\right ] \quad \mbox{and} \quad
M^\alpha_\tau=
{\tau}{\Gamma(\alpha +1}) I,
\]
where $e_k=d_{k+1}-d_{k}=(k+1)^{\alpha+1}+(k-1)^{\alpha+1}-2 k^{\alpha+1}$, and $I \in \mathbb{R}^{K \times K}$ is the identity matrix.
This formulation has been used in our implementation.

\section{Error Estimate or the FEM for Fractional ODEs}\label{sec:conv-ODE}

To illustrate the idea of error analysis, we first derive error estimates for fractional ODEs.
\subsection{Fractional ODE}
Consider the following fractional-order ODE: find $u(t)$ such that
\begin{equation}\label{eqn:ode}
_0\partial_t^\alpha  u + \lambda u = f, \quad \forall t \in (0,T), \quad \mbox{with } u(0)=0,
\end{equation}
where the constant $\lambda\in\mathbb{R}$ and $\lambda\ge0$.
The weak form reads: given $f\in L^2(0,T)$, find $u \in \widetilde{H}_L^\alpha(0,T) $
\begin{equation}\label{eqn:weak-ode}
   a_\lambda( u,\phi)  \equiv    ({_0\partial_t^\alpha}  u, \phi)_{L^2(0,T)}+ \lambda (u,\phi)_{L^2(0,T)}
= (f,\phi)_{L^2(0,T)} \quad \forall \phi\in L^2(0,T).
\end{equation}

By choosing $\phi={}_0\partial_t^\alpha v+\lambda v$ in $a_\lambda(\cdot,\cdot)$, since
$({}_0\partial_t^\alpha  v, v)_{L^2(0,T)}\ge 0$ for $v\in\widetilde{H}_L^\al(0,T)$, we deduce
\begin{equation*}
  \begin{aligned}
  a_\lambda(v,\phi) &= \|{}_0\partial_t^\alpha v\|_{L^2(0,T)}^2 + 2\lambda({}_0\partial_t^\alpha v,v)_{L^2(0,T)}+\lambda^2\|v\|^2_{L^2(0,T)}\\
   &\geq \tfrac{1}{2}(\|{}_0\partial_t^\alpha v\|_{L^2(0,T)} +\lambda\|v\|_{L^2(0,T)})^2.
  \end{aligned}
\end{equation*}
Since $ \|\phi\|_{L^2(0,T)} \le  \|{}_0\partial_t^\alpha v\|_{L^2(0,T)} +\lambda\|v\|_{L^2(0,T)}$,
we arrive at the desired inf-sup condition
\[
\|{}_0\partial_t^\alpha v\|_{L^2(0,T)}   + \lambda \|v\|_{L^2(0,T)}
\le 2 \sup_{\phi\in L^2(0,T) }\frac{a_\lambda(v,\phi)}{\quad \|\phi\|_{L^2(0,T)}}.
\]
For $\phi\in L^2(0,T)$, $\phi\neq0$, let $v={_0I_t^\alpha}\phi$.
By Lemma \ref{lem:pdf}, $a_\lambda(v,\phi)=(\phi,\phi)_{L^2(0,T)}+\lambda(\phi,v)_{L^2(0,T)}>0$.
Thus $a_\lambda(\cdot,\cdot)$
satisfies also a compatibility condition and problem \eqref{eqn:weak-ode} is well-posed.

Moreover, if $f\in \widetilde{H}_L^s(0,T)$, then the ODE \eqref{eqn:ode} has a
unique solution $u\in \widetilde{H}_L^{\alpha+s}(0,T)$ and
\begin{equation}\label{eqn:regularity-ode}
\| u \|_{\widetilde{H}_L^{\al+s}(0,T)} + \lambda\| u \|_{\widetilde{H}_L^s(0,T)}\le c \| f \|_{\widetilde{H}_L^s(0,T)},
\end{equation}
where the constant $c$ is independent of $\lambda$. This estimate follows directly from
Theorem \ref{thm:regularity} by identifying the operator $-\Delta$ with the scalar $\lambda$.

\begin{remark}\label{rem:ODEadjoint}
For the adjoint problem, to find   $w \in \widetilde{H}_R^\alpha(0,T) $ such that
$a_\lambda(\phi,w) =(\phi,f)_{L^2(0,T)}$ for all $ \phi \in L^2(0,T)$,
a similar inf-sup condition and regularity pick-up hold.
\end{remark}

With the spaces $\bVt$ and $\mathbb{W}_\tau$ defined in \eqref{eqn:spaceVt}, the Petrov-Galerkin FEM
for problem \eqref{eqn:ode} reads: given $f\in L^2(0,T)$, find $u_\tau \in \bVt$
such that
\begin{equation}\label{eqn:Galerkin-ode}
a_\lambda(u_\tau,\phi)=(f,\phi)_{L^2(0,T)} \quad \forall\phi\in \mathbb{W}_\tau.
\end{equation}
For any $v\in \bVt$, by letting $\phi=\phi_v={}_0\partial_t^\alpha v +\lambda\Pi_\tau v$ and
applying Lemma \ref{lem:pdf2}, and repeating the preceding argument, we derive the following
discrete inf-sup condition
\begin{equation}\label{ODEinfsup-discrete-imp}
\| v \|_{\widetilde{H}_L^\alpha(0,T)} +\lambda\|\Pi_\tau v\|_{L^2(0,T)}
\le c \sup_{\phi\in \mathbb{W}_\tau }\frac{a_\lambda(v ,\phi)}{\quad \|\phi\|_{L^2(0,T)}},
\end{equation}	
where $c$ is independent of $\lambda$. Thus problem
\eqref{eqn:Galerkin-ode} is well-posed and stable in  $\widetilde{H}_L^\alpha(0,T) $-norm.

\subsection{Properties of the fractional Ritz and fractionalized $L^2$-projections}\label{sec:frac-ritz}

For the analysis below, we define a fractional Ritz
projection $R_\tau^\al: \widetilde H_L^\alpha(0,T) \to  \bVt$  by
\begin{equation*}
({}_0\partial_t^\al R_\tau^\al v, \phi)_{L^2(0,T)} = ({}_0\partial_t^\al v, \phi)_{L^2(0,T)}  \quad \forall \phi \in \mathbb{W}_\tau.
\end{equation*}

The operator $R_\tau^\alpha $ has optimal approximation in both $\widetilde H^\alpha_L$- and $L^2$-norms.
\begin{lemma}\label{lem:fracRitz}
For the fractional Ritz projection $R_\tau^\alpha $, there holds
\begin{equation}\label{eqn:err-Ritz-proj}
\tau^\alpha\|{}_0\partial_t^\alpha(v-R_\tau^\alpha v)\|_{L^2(0,T)}+\|v -R_\tau^\al v\|_{L^2(0,T)}
      \le c \tau^{\al+s} \|v\|_{\widetilde{H}^{\al+s}_L(0,T)},\quad 0\leq s\leq 1.
\end{equation}
\end{lemma}
\begin{proof}
Let $e=v-R_\tau^\al v$.
Clearly, for any $v_\tau\in \mathbb{U}_\tau$,
\begin{equation*}
  ({_0\partial_t^\alpha}(R_\tau^\al v -v_\tau),\phi)_{L^2(0,T)}
     =({_0\partial_t^\alpha}(v-v_\tau), \phi)_{L^2(0,T)} \quad \forall \phi \in \mathbb{W}_\tau.
\end{equation*}
Upon taking $\phi={}_0\partial_t^\al(R_\tau^\al v -v_\tau)$ and by the Cauchy-Schwarz inequality, we have
\begin{equation*}
	\|{}_0\partial_t^\al e\|_{L^2(0,T)}\le  2 \inf_{v_\tau \in \bVt}\|{}_0\partial_t^\al(v-v_\tau)\|_{L^2(0,T)}.
\end{equation*}
By repeating the arguments of \cite[Lemma 4.2]{Jin2015petrovSINUM}, we obtain for $ 0 \le s  \le 1$
	\begin{equation}\label{eqn:frac-Ritz}
	\|{}_0\partial_t^\al e\|_{L^2(0,T)}\le c \tau^s
	\|v\|_{\widetilde{H}^{\al+s}_L(0,T)}.
\end{equation}
Now we prove the $L^2$-error bound.
Let $w\in \widetilde{H}_R^{\al}(0,T) $ be the solution to the adjoint problem
\begin{equation*}
(\phi, {_t\partial_T^\alpha}w)_{L^2(0,T)}= (\phi,e)_{L^2(0,T)} \quad \forall \phi \in L^2(0,T).
\end{equation*}
Similar to \eqref{eqn:regularity-ode}, cf. Remark \ref{rem:ODEadjoint},   the solution $w$ satisfies the a priori estimate
\begin{equation*}
  \|w\|_{\widetilde{H}_R^\al(0,T)} \le c \|e \|_{L^2(0,T)}.
\end{equation*}
Then by \eqref{eqn:adjoint} and Galerkin orthogonality, there holds
\begin{equation*}
\begin{aligned}
\|e\|^2_{L^2(0,T)}  & =(e,{_t\partial_T^\alpha} w)_{L^2(0,T)}= ({_0\partial_t^\alpha} e, w - w_\tau)_{L^2(0,T)}\\
&  \le  \|{}_0\partial_t^\al e\|_{L^2(0,T)} \inf_{w_\tau\in \mathbb{W}_\tau} \|w-w_\tau\|_{L^2(0,T)} \\
&\le c \tau^\alpha\|{}_0\partial_t^\al e\|_{L^2(0,T)}  \|w\|_{\widetilde{H}_R^{\al}(0,T)} \\
& \le c\tau^\alpha \|{_0\partial_t^\alpha}e\|_{L^2(0,T)}\|e\|_{L^2(0,T)}.
\end{aligned}
\end{equation*}
This together with \eqref{eqn:frac-Ritz} yields the desired error estimate.
\end{proof}

Next we introduce a fractionalized $L^2$-projection $P_\tau:L^2(0,T)\to \bVt$, defined by
\begin{equation*}
  (P_\tau v,\phi)_{L^2(0,T)}=(v,\phi)_{L^2(0,T)}\quad \forall\phi\in \mathbb{W}_\tau.
\end{equation*}
Let $b(\cdot,\cdot): \bVt\times \bWt\to \mathbb{R}$ by $b(v,\phi)= (v,\phi)_{L^2(0,T)}$.
For any $v\in \bVt$, choosing $\phi=\Pi_\tau v$ yields
$b(v,\phi)=(v,\Pi_\tau v)_{L^2(0,T)} =\|\Pi_\tau v\|_{L^2(0,T)}^2.$
This and Lemma \ref{lem:pdf2} yield the following inf-sup condition
\begin{equation*}
  \sup_{\phi\in \bWt}\frac{b(v,\phi)}{\|\phi\|_{L^2(0,T)}}\geq c \|v\|_{L^2(0,T)}.
\end{equation*}
Thus the projection operator $P_\tau$ is well defined. Next we study its approximation property.

\begin{lemma}\label{lem:fracL2}
For the fractionalized $L^2$-projection $P_\tau$, there holds
\begin{equation*}
  \begin{aligned}
  \|v-P_\tau v\|_{L^2(0,T)} &\leq c\tau^s\|v\|_{\widetilde H_L^s(0,T)},\quad 0\leq s\leq \alpha+1,\\
  \|v-P_\tau v\|_{\widetilde H_L^\alpha(0,T)} &\leq c\tau^s\|v\|_{\widetilde H_L^{\alpha+s}(0,T)},\quad 0\leq s\leq 1.
  \end{aligned}
\end{equation*}
\end{lemma}
\begin{proof}
Clearly, by Lemma \ref{lem:pdf2}, it is stable in $L^2(0,T)$, i.e.,
$\|P_\tau v\|_{L^2(0,T)}\leq c\|v\|_{L^2(0,T)}.$
It directly follows the inf-sup condition that
\begin{equation*}
  \|v-P_\tau v\|_{L^2(0,T)}\leq c\inf_{v_\tau\in \bVt}\|v-v_\tau\|_{L^2(0,T)}.
\end{equation*}
In particular, if $v\in\widetilde H^s_L(0,T)$, with $s\geq\alpha$, we may take $v_\tau=R_{\tau}^\alpha v$ to deduce
\begin{equation*}
  \|v-P_\tau v\|_{L^2(0,T)}\leq c\tau^s\|v\|_{\widetilde H_L^s(0,T)}\quad \alpha\leq s\leq \alpha+1.
\end{equation*}
This estimate, the $L^2$-stability and interpolation yield the first estimate.
Next, by the triangle inequality, we derive the $\widetilde H_L^\alpha$-estimate:
\begin{equation*}
  \begin{aligned}
    \|v-P_\tau v\|_{\widetilde H_L^\alpha(0,T)} & \leq \|v-R_\tau^\alpha v\|_{\widetilde H_L^{\alpha}(0,T)}+\|R_\tau^\alpha v-P_\tau v\|_{\widetilde H_L^{\alpha}(0,T)}\\
       & \leq (1+\|P_\tau\|_{\widetilde H_L^\alpha(0,T)\to \widetilde H_L^\alpha(0,T)})\|v-R_\tau^\alpha v\|_{\widetilde{H}_L^{\alpha}(0,T)} \\
      &    \leq c\tau^s\|v\|_{\widetilde H_L^{\alpha+s}(0,T)},
  \end{aligned}
\end{equation*}
where the last inequality follows by the
$\widetilde H_L^\alpha$-stability of $P_\tau$ from Lemma \ref{lem:stability} below.
\end{proof}

The next result gives the $\widetilde H_L^\alpha$-stability of $P_\tau$, which is needed in the proof of Lemma \ref{lem:fracL2}.
\begin{lemma}\label{lem:stability}
The fractionalized $L^2$-projection $P_\tau$ is stable on $\widetilde H_L^\alpha(0,T)$.
\end{lemma}
\begin{proof}
First, we show the inverse estimate
\begin{equation}\label{eqn:inversest}
  \|v\|_{\widetilde H_L^\alpha(0,T)}\leq c\tau^{-\alpha}\|v\|_{L^2(0,T)}\quad \forall v\in \bVt.
\end{equation}
For any $v\in \bVt$, there exists $\phi\in \bWt$
such that $v={_0I_t^\alpha}\phi$. Thus it is equivalent to
\begin{equation*}
  \|\phi\|_{L^2(0,T)}\leq c\tau^{-\alpha}\|{_0I_t^\alpha}\phi\|_{L^2(0,T)}\quad \forall \phi\in \bWt.
\end{equation*}
Further, by \eqref{eqn:adjoint} and norm equivalence, we have
\begin{equation*}
\begin{aligned}
    \|\phi\|_{H^{-\alpha}(0,T)} &
    \equiv \sup_{\psi\in \widetilde{H}_L^\alpha(0,T)}\frac{(\phi,\psi)_{L^2(0,T)}}{\|\psi\|_{\widetilde{H}_L^\alpha(0,T)}} \\
   & = \sup_{\psi\in \widetilde{H}_L^\alpha(0,T)}
           \frac{({_0I_t^\alpha}\phi,{_t\partial_T^\alpha}\psi)_{L^2(0,T)}}{\|\psi\|_{\widetilde{H}_L^\alpha(0,T)}}
           \leq c \|{_0I_t^\alpha}\phi\|_{L^2(0,T)}.
\end{aligned}
\end{equation*}
Now the following inverse estimate is known \cite[Theorem 4.6]{Dahmen:2003} (see also Remark \ref{eq:inv-simple} below)
\begin{equation}\label{eq:dahmen}
   \|\phi\|_{L^2(0,T)}\leq c\tau^{-\alpha}\|\phi\|_{H^{-\alpha}(0,T)}\quad \forall \phi\in \bWt,
\end{equation}
which directly yields \eqref{eqn:inversest}. Now it follows from
\eqref{eqn:inversest} that for any $v\in \widetilde H_L^\alpha(0,T)$
\begin{equation*}
  \begin{aligned}
  \|P_\tau v\|_{\widetilde H_L^\alpha(0,T)}
  &\leq \|R_\tau^\alpha v\|_{\widetilde H_L^\alpha(0,T)} + \|R_\tau^\alpha v-P_\tau v\|_{\widetilde H_L^\alpha(0,T)}\\
  &\leq c\|v\|_{\widetilde H_L^\alpha(0,T)} + c\tau^{-\alpha}\|R_\tau^\alpha v-P_\tau v\|_{L^2(0,T)}\leq c\|v\|_{\widetilde H_L^\alpha(0,T)},
  \end{aligned}
\end{equation*}
where the last inequality follows by the $L^2$-estimates for $\Pi^\alpha_\tau$ and $P_\tau$,
cf. Lemmas \ref{lem:fracRitz} and \ref{lem:fracL2}.
\end{proof}

\begin{remark}\label{eq:inv-simple}
The inverse inequality \eqref{eq:dahmen} is a special case
of a general result in \cite{Dahmen:2003}.
In our case, it follows easily from a duality argument. For a given $\phi \in \bWt$,
find $v_\phi \in \widetilde H^1_L(0,T)$ such that $( v_\phi^\prime, \varphi^\prime)_{L^2(0,T)}=(\phi, \varphi)_{L^2(0,T)}$
for all $\varphi \in \widetilde H_L^1(0,T)$. Then $ \| v_\phi^\prime \|_{L^2(0,T)} \le c\| \phi\|_{H^{-1}(0,T)}$
and $\| v_\phi^{\prime \prime} \|_{L^2(0,T)} = \| \phi\|_{L^2(0,T)}$. Since $\phi\in \bWt$,
$ v_\phi^\prime $ is conforming piecewise linear and
the inverse inequality $ \| v_\phi^{\prime \prime} \|_{L^2(0,T)} \le c \tau^{-1} \| v_\phi^\prime \|_{L^2(0,T)}$ holds. Consequently,
$$
\| \phi\|_{L^2(0,T)} \le c \tau^{-1}  \| v_\phi^\prime \|_{L^2(0,T)} \le c  \tau^{-1}  \| \phi \|_{H^{-1}(0,T)},
$$
and by interpolation, we obtain the desired inequality  \eqref{eq:dahmen} for $\alpha\in(0,1)$.
\end{remark}

\subsection{Error estimates for fractional ODEs}\label{sec:error-ODE}

Next, we derive error estimates for
the Galerkin scheme \eqref{eqn:Galerkin-ode}  for the fractional ODE \eqref{eqn:ode}. We first establish the following result:
\begin{theorem}\label{thm:ODEenergy}
Let $f \in \widetilde{H}_L^{s}(0,T)$. Then the solution $u_\tau \in \bVt$ of problem \eqref{eqn:Galerkin-ode}
satisfies
\begin{equation*}
\|{}_0\partial_t^\al( u-u_\tau)\|_{L^2(0,T)} + \lambda\| u-u_\tau  \|_{L^2(0,T)}  \le c \tau^s
       \| f \|_{\widetilde{H}_L^{s}(0,T)}.
\end{equation*}
\end{theorem}
\begin{proof}
Since $({}_0\partial_t^\alpha  u, u)_{L^2(0,T)}\ge 0$, cf. Lemma \ref{lem:pdf2}, repeating
the arguments in Section \ref{sec:frac-ritz} yields
\begin{equation*}
 \|{}_0\partial_t^\al( u-u_\tau)\|_{L^2(0,T)}  + \lambda\| u-u_\tau  \|_{L^2(0,T)}
 \le 2 \inf_{v \in \bVt}(\|{}_0\partial_t^\al (u-v)\|_{L^2(0,T)} +\lambda \|{u-v}\|_{L^2(0,T)}).
\end{equation*}
Taking $v=P_\tau u$ and appealing to Lemma \ref{lem:fracL2} and
\eqref{eqn:regularity-ode} yield the desired estimate.
\end{proof}

Next we establish an error bound in $L^2$-norm.
\begin{theorem}\label{thm:ODEL2}
Let $f \in \widetilde{H}_L^{s}(0,T)$. Then the solution $u_\tau$ of problem \eqref{eqn:Galerkin-ode}
satisfies
\begin{equation*}
\| u-u_\tau \|_{L^2(0,T)} \le c   \tau^{\alpha+s}     
       \| f\|_ {\widetilde{H}_L^{s}(0,T)}.
\end{equation*}
\end{theorem}
\begin{proof}
We apply a duality argument.
Let $ z \in \widetilde{H}_R^{\al}(0,T) $ be the solution to (with $e=u-u_\tau$)
\begin{equation*}
(\phi, {_t\partial_T^\alpha} z)_{L^2(0,T)}+ \lambda(\phi,  z)_{L^2(0,T)} = (\phi, e)_{L^2(0,T)} \quad \forall \phi \in L^2(0,T).
\end{equation*}
Then by the adjoint stability, cf. Remark \ref{eq:inv-simple},
we have the following \textit{a priori} estimate:
\begin{equation*}
   \|z\|_{\widetilde{H}_R^\al(0,T)} + \lambda \|z\|_{L^2(0,T)}\le c \|e\|_{L^2(0,T)}.
\end{equation*}
By \eqref{eqn:adjoint} and Galerkin orthogonality, we deduce for any $z_\tau\in \mathbb{W}_\tau$
\begin{equation*}
\begin{aligned}
  \|e\|_{L^2(0,T)}^2 = a(e, z - z_\tau)
 &  \le \big( \|{}_0\partial_t^\al e\|_{L^2(0,T)}+\lambda\|e\|_{L^2(0,T)}\big) \inf_{z_\tau\in \mathbb{W}_\tau} \|z-z_\tau\|_{L^2(0,T)} \\
 &  \le c \tau^\alpha \|z\|_{\widetilde{H}_R^{\al}(0,T)} \big( \|{}_0\partial_t^\al e\|_{L^2(0,T)}+\lambda\|e\|_{L^2(0,T)}\big).
 \end{aligned}
 \end{equation*}
Now using the bound on $z$ and Theorem \ref{thm:ODEenergy} completes the proof.
\end{proof}

\begin{remark}
Both $\widetilde H_L^\alpha$- and $L^2$-estimates are independent
of the parameter $\lambda$, concurring with the inf-sup condition \eqref{ODEinfsup-discrete-imp}.
In either norm, the convergence rate is of optimal order.
\end{remark}

\subsection{Enhanced error estimates for  $f \in H^s(0,T),  \frac{1}{2}<s\leq 1$}\label{sec:improvedODE}

The trial space $\bVt$ allows improving the error estimates.
First we consider the special case of a source term $f \equiv 1$. Clearly, $f \in {\widetilde H}^{\beta}_L(0,T)$
for any $\beta<\frac12$, and thus $u \in \widetilde H_L^{\alpha+\beta}(0,T)$. Theorem \ref{thm:ODEL2}
gives an $L^2$-error with the rate $O(\tau^{\alpha +\beta})$. By Laplace transform, we derive
$
u(t)= t^\alpha E_{\alpha,\alpha+1}(-\lambda t^\alpha).
$
In the splitting $u=\frac{t^\alpha}{\Gamma(\alpha +1)} + \tilde{u}$, since
$\tilde{u} \in \widetilde{H}_L^{2\alpha+\beta}(0,T)$
and $t^\alpha \in  \bVt$, we obtain
\begin{equation*}
  \inf_{v \in  \mathbb{V}_\tau} \| u - v\|_{\widetilde{H}_L^\alpha(0,T)} = \inf_{v \in  \mathbb{V}_\tau} \| \tilde u - v\|_{\widetilde{H}_L^\alpha(0,T)} \le c \tau^{\min(\alpha +\beta,1)},\quad \forall\beta\in[0,\tfrac12).
\end{equation*}
Then by a duality argument, we have
\begin{equation*}
\| u - u_\tau\|_{L^2(0,T)} \le c \tau^{\alpha+ \min(\alpha +\beta,1)},\quad \forall\beta\in[0,\tfrac12).
\end{equation*}

Generally, for $f \in H^s(0,T)$, $\frac 12 < s \le 1$, one may split $f = f(0) + \tilde f$, with
$\tilde f = f-f(0) \in \widetilde{H}_L^{s}(0,T)$, and accordingly, $u=\hat u + \tilde u$, where
$_0\partial_t^\alpha  \hat u + \lambda \hat u = f(0)$ and
$ _0\partial_t^\alpha  \tilde u + \lambda \tilde u = \tilde f (t),$
with $\hat u(0)=\tilde u(0)=0$. By \eqref{eqn:regularity-ode},
$\tilde u  \in \widetilde{H}_L^{\al+s}(0,T)$ and can be approximated with an $L^2$-error $O(\tau^{\alpha + s})$. Hence we have
\begin{equation*}
\| u - u_\tau \|_{L^2(0,T)} + \tau^\alpha\| u - u_\tau \|_{\widetilde{H}_L^\alpha(0,T)} \leq c \tau^{\alpha+ \min(\alpha + \beta,s)},
\quad\forall\beta\in[0,\tfrac12).
\end{equation*}
This improves the error estimates in Theorems \ref{thm:ODEenergy} and \ref{thm:ODEL2} (for $f(0)\neq0$).

\section{Error Estimates of the FEM for Fractional PDE}\label{sec:conv-PDE}

Now we derive error estimates for the space-time scheme \eqref{eqn:BV-weak-h}. To this end, we
recall the semidiscrete Galerkin problem for problem \eqref{orig}: for $ t \in (0,T]$ find $u_h(t) \in \bVh $ such that
\begin{equation}\label{eqn:semidisc}
( {}_0 \partial_{t}^\alpha u_h(t),\phi)_{L^2(\Omega)} +(\nabla u_h(t),\nabla\phi)_{L^2(\Omega)}=(f(t), \phi)_{L^2(\Omega)}
                \quad \forall \phi \in \bVh,  
\end{equation}
with $u_h(0)=0$. Next we recast it into a space semidiscrete space-time formulation by
defining a trial space $\bBalh :=  \widetilde H_L^\alpha( 0,T)\otimes \mathbb{X}_h \subset \Bal(Q_T)$
and test space $ \bVxh :=  L^2(0,T)\otimes \mathbb{X}_h  \subset  V(Q_T)$,
endowed with the associated norms on $\Bal(Q_T)$
and $V(Q_T)$, respectively. Then problem \eqref{eqn:semidisc} is equivalent to: find $u_h\in \bBalh$ such that
$$
a(u_h, \phi) = \langle f, \phi \rangle_{L^2(Q_T)}, \quad \forall  \phi \in \bVxh.
$$

The argument of Lemma \ref{lem:inf-sup} similarly yields an inf-sup condition for the semidiscrete
problem: there holds for some $c$ independent of $h$
\begin{equation}\label{eqn:inf-sup-semi}
	\sup_{\phi \in \bVxh} \frac{a(v,\phi) }{\|\phi \|_V} \ge c\|v \|_{\Bal(Q_T)}\quad \forall v\in \bBalh.
\end{equation}

Using the basis $\{\psi_j\}_{j=1}^N$ (cf. Section \ref{sec:FEM_stability}),
we expand the solutions $u_h$ and $u_{h \tau}$ into
\begin{equation}\label{eqn:expansion}
u_h(t)=\sum_{j=1}^Nu_{j,h}(t)\psi_j\quad \mbox{and} \quad u_{h \tau}(t)=\sum_{j=1}^Nu_{j,h \tau}(t)\psi_j,
\end{equation}
where $u_{j,h}(t)=(u_h(t),\psi_j)_{L^2(\Omega)}$ and $u_{j,h \tau}(t)=(u_{h \tau}(t),\psi_j)_{L^2(\Omega)}$.
Further, the function $u_{j,h}(t)$ satisfies $u_{j,h}(0)=0$ and
\begin{equation*}
_0\partial_t^\alpha u_{j,h} + \lambda_{j,h}u_{j,h} = f_{j,h},\quad 0<t\leq T,
\end{equation*}
where $f_{j,h}(t)=(P_hf(\cdot,t),\psi_j)_{L^2(\Omega)}\in \widetilde{H}_L^s(0,T)$, if
$f\in \widetilde{H}_L^s(0,T;L^2(\Omega))$. Similarly, the function $u_{j,h \tau}\in \bVt$ satisfies
\begin{equation*}
(_0 \partial_{t}^\alpha u_{j,h \tau}+\lambda_{j,h}u_{j,h \tau}, \phi)_{L^2(0,T)}=(f_{j,h},\phi)_{L^2(0,T)},
               \quad \forall\phi\in \mathbb{W}_\tau.
\end{equation*}
In other words, $u_{j,h,\tau}$ is the Petrov-Galerkin approximation
of $u_{j,h}$, and thus Theorems \ref{thm:ODEenergy} and \ref{thm:ODEL2}
give the following error estimates on $e_{j,h}:=u_{j,h}-u_{j,h \tau}$:
\begin{align}
\|{}_0\partial_t^\al e_{j,h}\|_{L^2(0,T)} + \lambda_{j,h}\| e_{j,h}\|_{L^2(0,T)} & \le c \tau^s
\| f_{j,h} \|_{\widetilde{H}_L^{s}(0,T)},\label{eqn:energy-uh}\\
\|e_{j,h}\|_{L^2(0,T)}&\leq c\tau^{\alpha+s}\|f_{j,h}\|_{\widetilde{H}^s_L(0,T)},\label{eqn:est-L2-uh}
\end{align}
where the constant $c$ is independent of $h$.
Next we give an energy estimate for the semidiscrete approximation $u_h$.
\begin{lemma}\label{lem:semi-est-B}
For $f \in L^2 ( Q_T)$, the semidiscrete solution $u_h$  satisfies
\begin{equation*}
	\|{u_h-u}\|_{\Bal(Q_T)}\le c h\|f\|_{L^2 (Q_T)}.
\end{equation*}
\end{lemma}
\begin{proof}
By \eqref{neg-norm-Ph} and Theorem \ref{thm:regularity}, we have for $\varrho= P_h u -u$
\begin{equation*}
	\|\varrho\|_{B_L^\alpha(Q_T)}
\le c h( \| u\|_{\widetilde{H}^{\al}_L(0,T;L^2(\Omega))}
     + \| u\|_{L^2(0,T; H^2(\Omega))}) \le c h \|f\|_{L^2(Q_T)}.
\end{equation*}
The function $\vartheta := u_h-P_h u$ satisfies $\vartheta(0)=0$ and $
{}_0\partial_t^\alpha\vartheta-\Delta_h\vartheta=\Delta_h(P_hu-R_hu)=\Delta_h R_h \varrho$
(with $R_h$ being the Ritz projection) \cite[equation (3.22)]{JinLazarovZhou:2013}.
By \eqref{eqn:inf-sup-semi} and \eqref{hm1hm1h}, we have
	\begin{equation*}
	\|\vartheta\|_{B^\al_L(Q_T)} \le c \| \Delta_h  R_h \varrho\|_{L^2(0,T;H^{-1}(\Omega))}
	\le c \| R_h \varrho\|_{L^2(0,T;H^1(\Omega))}\le c h\|f\|_{L^2(Q_T)}.
	\end{equation*}
These two estimates and the triangle inequality complete the proof.
\end{proof}

Then we can derive an energy norm estimate for the scheme \eqref{eqn:BV-weak-h}.
\begin{theorem}\label{thm:err-energy}
Let $f \in \widetilde H_L^{s}(0,T;L^2(\Omega ))$ with $0 \le s \le 1$, and $u$ and
$u_{h \tau}$ be the solutions of \eqref{eqn:BV-weak} and
\eqref{eqn:BV-weak-h}, respectively. Then there holds
\begin{equation*}
	\| u-u_{h \tau}\|_{\Bal(Q_T)}
	\le c ( \tau^s+ h)\|f\|_{\widetilde{H}_L^s( {0,T;L^2 (\Omega )})}.
\end{equation*}
\end{theorem}
\begin{proof}
By the expansion \eqref{eqn:expansion} and the estimate \eqref{eqn:energy-uh}, we bound the error $e_h:=u_h-u_{h \tau}$ by
\begin{equation*}
  \begin{aligned}
\|e_h\|_{\Bal(Q_T)}^2
  & =\sum^N_{j=1}\lambda_{j,h}^{-1}\|{}_0\partial_t^\alpha e_{j,h}\|_{L^2(0,T)}^2+\sum_{j=1}^N\lambda_{j,h}\|e_{j,h}\|_{L^2(0,T)}^2\\
  &\le c \tau^{2s} \sum^N_{j=1}\lambda_{j,h}^{-1}  \| f_{j,h} \|^2_{\widetilde{H}_L^{s}(0,T)}
    \leq c \tau^{2s} \|f \|^2_{\widetilde{H}_L^{s}(0,T;H^{-1}(\Omega))}.
  \end{aligned}
\end{equation*}
This, Lemma \ref{lem:semi-est-B} and the triangle inequality give the desired assertion.
\end{proof}

Finally,  we present the $L^2(Q_T)$ error estimate.
\begin{theorem}\label{thm:err-L2}
	For $f \in \widetilde H_L^{s} \left( {0,T;L^2(\Omega)} \right)$, let $u$ and $u_{h \tau}$
	be the solutions of \eqref{eqn:BV-weak}
	and \eqref{eqn:BV-weak-h}, respectively. Then there holds
	\begin{equation*}
	\| u-u_{h \tau} \|_{{L^2}({Q_T})}\le c ({\tau ^{\alpha+s}} +  h^2 )\| f\|_{\widetilde{H} _L^s(0,T;L^2(\Omega))}.
	\end{equation*}
\end{theorem}
\begin{proof}
By \eqref{eqn:est-L2-uh} and the $L^2(\Omega)$-stability of $P_h$, we can bound the error $e_h:=u_h-u_{h \tau}$ by
\begin{equation*}
\begin{aligned}
	\|e_h\|_{L^2(Q_T)}^2=\sum^N_{j=1}\|e_{j,h}\|_{L^2(0,T)}^2
&	\leq c\tau^{2(\alpha+s)}\sum_{j=1}^N\|f_{j,h}\|_{\widetilde{H}_L^s(0,T)}^2 \\
&     \leq c\tau^{2(\alpha+s)}\|f\|_{\widetilde{H}_L^s(0,T;L^2(\Omega))}^2,
\end{aligned}
\end{equation*}
By \cite[Theorem 3.4]{JinLazarovPasciakZhou:2015}, there holds
$\|{u_h-u}\|_{L^2(Q_T)}\le c h^2\|f\|_{L^2 (Q_T)},
$ which completes the proof.
\end{proof}

\begin{remark}\label{rem:Caputo}
In practice, the Caputo fractional derivative is preferred,
since it allows specifying initial conditions as usual, cf. \cite[pp. 353--358]{KilbasSrivastavaTrujillo:2006}.
Thorough discussion about the choice of initial conditions and their mathematical correctness one can find in
\cite{li2016convolution}.
The approach can be extended to the case of smooth initial data:
\begin{align*}
\left\{
\begin{aligned}
&\partial_t^\alpha u-\Delta u=f  &&\mbox{in}\,\,\,\Omega\times(0,T),\\
&u=0 &&\mbox{on}\,\,\partial\Omega\times(0,T),\\
&u(0)=u_0 &&\mbox{in}\,\, \Omega,
\end{aligned}
\right.
\end{align*}
where $u_0\in H_0^1(\Omega)\cap H^2(\Omega)$  and
$\partial_t^\alpha u :={\,_0I^{1-\alpha}_t}u^\prime$
denotes the Caputo derivative of order  $\alpha\in(0,1)$. The function $w:=u-u_0$ satisfies \eqref{orig}
with a source $ F = f + \Delta u_0$, for which our approach applies.
Further, the extension to general elliptic operators and boundary conditions is
direct.
\end{remark}

\section{Numerical examples}\label{sec:numerics}
Now we numerically illustrate our theoretical findings.
Since the semidiscrete problem has been verified \cite{JinLazarovPasciakZhou:2015},
we focus on the temporal discretization error below. In all tables the computed rates are given in the
last column, whereas the numbers in brackets denote the theoretical rates.

\subsection{Fractional ODEs}
First we examine the convergence of the method for fractional ODEs. We
consider the initial value problem
\begin{equation*}
{_0\partial_t^\alpha} u(t) + u(t) = e^t\quad \text{in }(0,T), \quad \text{with} ~~ u(0)=0.
\end{equation*}
The source term $f(t)= e^t $ belongs to the space $H^1(0,T)$, and also in $ \widetilde H_L^s(0,T)$ for $s < \frac12$.
Thus, Theorems \ref{thm:ODEenergy} and \ref{thm:ODEL2} give a convergence rate $O(\tau^s)$ in the $H^\alpha(0,T)$-norm,
and $O(\tau^{\alpha+s})$ in the $L^2(0,T)$-norm, respectively.
By the discussion in Section \ref{sec:improvedODE}, we have an improved convergence rate
$O(\tau^{\min( 1, \alpha+s)}) $ in the $\widetilde{H}^\alpha_L(0,T)$-norm and $O(\tau^{\alpha +\min( 1,
\alpha+s)}) $ in the $L^2(0,T)$-norm, $s\in(0,\frac12)$.
These improved rates are numerically confirmed
by Table \ref{tab:ODE}, where the reference solution is computed on a much finer mesh with a time step size $\tau=1/2000$.

\begin{table}[hbt!]
\caption{The  errors $\|u -u_\tau \|_{L^2(0,T)}/\|u\|_{L^2(0,T)}$ and $\|u - u_\tau\|_{H^\alpha(0,T)}/\|u\|_{L^2(0,T)}$
for the fractional ODE with $\alpha=0.3$, $0.5$, $0.7$ and $0.9$.}\label{tab:ODE}
	\centering
	\begin{tabular}{|c|c|cccccc|c|}
		\hline
		$\alpha$ & $ K$ &$10$ &$20$ &$40$ & $80$ & $160$ & $320$ &rate \\
		\hline
		0.3 & $L^2$      & 8.49e-3 & 3.96e-3 & 1.92e-3 & 9.57e-4 & 4.68e-4 & 2.36e-4 & 1.03 (1.10)\\
            & $H^\alpha$ & 3.15e-2 & 1.78e-2 & 1.04e-2 & 6.18e-3 & 3.75e-3 & 2.33e-3 & 0.75 (0.80)\\
        \hline
		0.5 & $L^2$      & 3.88e-3 & 1.51e-3 & 5.89e-4 & 2.29e-4 & 8.74e-5 & 3.37e-5 & 1.36 (1.50)\\
            & $H^\alpha$ & 3.20e-2 & 1.74e-2 & 9.48e-3 & 5.12e-3 & 2.78e-3 & 1.54e-3 & 0.87 (1.00)\\
        \hline
		0.7 & $L^2$      & 1.66e-3 & 5.15e-4 & 1.59e-4 & 4.94e-5 & 1.52e-5 & 4.73e-6 & 1.69 (1.70)\\
            & $H^\alpha$ & 2.98e-2 & 1.53e-2 & 7.81e-3 & 3.96e-3 & 2.01e-3 & 1.04e-3 & 0.96 (1.00)\\
        \hline
		0.9 & $L^2$      & 8.51e-4 & 2.21e-4 & 5.74e-5 & 1.49e-5 & 3.91e-6 & 1.03e-6 & 1.93 (1.90)\\
            & $H^\alpha$ & 2.84e-2 & 1.42e-2 & 7.12e-3 & 3.56e-3 & 1.78e-3 & 9.10e-4 & 0.99 (1.00)\\
		\hline
   \end{tabular}
\end{table}

\subsection{1-D fractional PDEs}
Now we consider examples on the unit interval $\Omega=(0,1)$ with $T=1$, and
perform numerical tests on the following four sets of problem data:
\begin{enumerate}
	\item[(a)] $f(x,t)= x(1-x)(e^t-1)$ is very smooth in time, with $e^t-1\in \widetilde{H}_L^1(0,T)$.
	\item[(b)] $f(x,t)= x(1-x) e^t$ is mildly smooth in time, with $e^t\in \widetilde{H}_L^s(0,T) \cap H^1(0,T)$, $s<1/2$.
    \item[(c)] $f(x,t)= t^{-0.3}x(1-x)$ is  nonsmooth in time, with $t^{-0.3}\in \widetilde{H}_L^s(0,T)$, $s<0.2$.
    \item[(d)] $f(x,t)= t^{-0.3}$ is nonsmooth in time, with $t^{-0.3}\in \widetilde{H}_L^s(0,T)$, $s<0.2$.
\end{enumerate}

In cases (a)--(c), the source $f$ is compatible with the zero initial data, but not case (d).
In our computation, we fix the spatial mesh size $h$ at $h=1/2000$. The reference solutions are
computed on a finer temporal mesh with a time step size $\tau=1/2000$.  The numerical results are given in Table \ref{tab:exp1D}.
The empirical $L^2(Q_T)$ convergence rate agrees well with the theoretical one $O(\tau^{s+\alpha})$,
cf. Theorem \ref{thm:err-L2}. The (temporal) convergence improves steadily with the temporal regularity of the source  $f$
and for a fixed $f$, with the fractional order $\alpha$, reflecting the improved temporal solution regularity.
It is also worth noting that the spatial regularity of the source $f$ does not influence the
temporal convergence, which concurs with Theorem \ref{thm:err-L2}. Further, for case (b) with large
fractional order $\alpha$, e.g., $\alpha=0.7$ or $0.9$, we observe an empirical convergence rate higher than the theoretical
one $O(\tau^{s+\alpha})$. This phenomenon is analogous to that for fractional ODEs in Section \ref{sec:improvedODE},
due to the special construction of the trial space $\bBal$, and might be analyzed as in the ODE case, which, however, is beyond the scope of this work.

In Table \ref{tab:exp2-nodes}, we present the $L^2(\Omega)$-error at the final time $T$ for examples (c) and (d), by viewing
the space-time method \eqref{eqn:BV-weak-h} as a time-stepping scheme. Numerically one observes an $O(\tau^{\alpha+1})$ rate,
irrespective of the spatial regularity of the source $f$. The precise mechanism for the
high convergence rate in the case of nonsmooth data is to be studied.

\begin{table}[hbt!]
\centering
\caption{The relative error $\|u-u_{h \tau}\|_{L^2(Q_T)}/\|u\|_{L^2(Q_T)}$  for examples (a)--(d)
with $\alpha=0.3$, $0.5$, $0.7$, $0.9$, and $h=1/2000$.}\label{tab:exp1D}
	\begin{tabular}{|c|c|cccccc|c|}
		\hline
	case & $\alpha\backslash K $  &$10$ &$20$ &$40$ & $80$ & $160$ &$320$ &rate \\
		\hline
		\multirow{3}{*}{(a)}& $0.3$  & 1.85e-2 & 7.50e-3 & 3.07e-3 & 1.27e-3 & 5.17e-4 & 2.12e-4 & 1.28 (1.30) \\
		                    & $0.5$  & 8.95e-3 & 3.16e-3 & 1.12e-3 & 4.03e-4 & 1.42e-4 & 5.05e-5 & 1.49 (1.50) \\
		                    & $0.7$  & 4.71e-3 & 1.45e-3 & 4.44e-4 & 1.36e-4 & 4.12e-5 & 1.26e-5 & 1.70 (1.70) \\
		                    & $0.9$  & 2.84e-3 & 7.60e-4 & 2.00e-4 & 5.27e-5 & 1.38e-5 & 3.65e-6 & 1.92 (1.90) \\
		\cline{1-9}
		\multirow{3}{*}{(b)}&$0.3$   & 2.66e-2 & 1.76e-2 & 1.14e-2 & 7.33e-3 & 4.48e-3 & 2.77e-3 & 0.65 (0.80) \\
		                    &$0.5$   & 2.87e-2 & 1.58e-2 & 8.17e-3 & 3.99e-3 & 1.81e-3 & 8.20e-4 & 1.02 (1.00) \\
		                    &$0.7$   & 2.21e-2 & 8.94e-3 & 3.28e-3 & 1.12e-3 & 3.67e-4 & 1.18e-4 & 1.50 (1.20) \\
		                    &$0.9$   & 1.23e-2 & 3.53e-3 & 9.64e-4 & 2.57e-4 & 6.76e-5 & 1.77e-5 & 1.88 (1.40) \\
		\cline{1-9}
		\multirow{3}{*}{(c)}&$0.3$   & 2.90e-1 & 2.36e-1 & 1.90e-1 & 1.52e-1 & 1.18e-1 & 9.08e-2 & 0.33 (0.50) \\
		                    &$0.5$   & 2.44e-1 & 1.73e-1 & 1.18e-1 & 7.90e-2 & 5.10e-2 & 3.27e-2 & 0.58 (0.70) \\
		                    &$0.7$   & 1.80e-1 & 1.01e-1 & 5.48e-2 & 2.89e-2 & 1.51e-2 & 8.01e-3 & 0.89 (0.90) \\
		                    &$0.9$   & 1.10e-1 & 4.92e-2 & 2.15e-2 & 9.55e-3 & 4.29e-3 & 1.95e-3 & 1.16 (1.10) \\
		\cline{1-9}
		
		\multirow{3}{*}{(d)}&$0.3$   & 2.90e-1 & 2.36e-1 & 1.90e-1 & 1.52e-1 & 1.18e-1 & 9.10e-2 & 0.33 (0.50) \\
		                    &$0.5$   & 2.45e-1 & 1.73e-1 & 1.19e-1 & 7.96e-2 & 5.16e-2 & 3.34e-2 & 0.57 (0.70) \\
		                    &$0.7$   & 1.81e-1 & 1.02e-1 & 5.59e-2 & 2.99e-2 & 1.60e-2 & 8.66e-3 & 0.87 (0.90) \\
		                    &$0.9$   & 1.12e-1 & 5.11e-2 & 2.31e-2 & 1.05e-2 & 4.81e-3 & 2.21e-3 & 1.13 (1.10) \\
		\hline
	\end{tabular}
\end{table}

\begin{table}[hbt!]
\caption{The relative error $\|u(\cdot,T) - u_{h \tau}(\cdot,T)\|_{L^2(\Omega)}/\|u(\cdot,T)\|_{L^2(\Omega)}$
at the time $T=1$ for examples (c)-(d) with $\alpha=0.3$, $0.5$, $0.7 $ and $0.9$, and $h=1/2000$. }\label{tab:exp2-nodes}
	\centering
		\begin{tabular}{|c|c|cccccc|c|}
			\hline
			case& $\alpha\backslash K $   &$10$ &$20$ & $40$ & $80$ &$160$ &$320$& rate\\
			\hline
			\multirow{3}{*}{(c)}& $0.3$  & 6.20e-3 & 2.46e-3 & 9.83e-4 & 3.92e-4 & 1.54e-4 & 5.91e-5 & 1.34 $(--)$ \\
			&$0.5$                       & 2.26e-3 & 7.82e-4 & 2.71e-4 & 9.39e-5 & 3.23e-5 & 1.08e-5 & 1.54 $(--)$\\
			&$0.7$                       & 4.37e-4 & 1.36e-4 & 4.16e-5 & 1.26e-5 & 3.82e-6 & 1.13e-6 & 1.71 $(--)$\\
			&$0.9$                       & 3.13e-4 & 6.68e-5 & 1.57e-5 & 3.63e-6 & 7.96e-7 & 1.60e-7 & 2.18 $(--)$\\
            \cline{1-9}
			\multirow{3}{*}{(d)}&$0.3$   & 6.20e-3 & 2.46e-3 & 9.83e-4 & 3.92e-4 & 1.54e-4 & 5.91e-5 & 1.34 $(--)$\\
			                    &$0.5$   & 2.26e-3 & 7.82e-4 & 2.70e-4 & 9.38e-5 & 3.23e-5 & 1.08e-5 & 1.54 $(--)$\\
			                    &$0.7$   & 4.54e-4 & 1.36e-4 & 4.16e-5 & 1.26e-5 & 3.81e-6 & 1.13e-6 & 1.73 $(--)$\\
                                &$0.9$   & 2.46e-3 & 1.19e-4 & 1.58e-5 & 3.63e-6 & 7.98e-7 & 1.60e-7 & 2.78 $(--)$\\
			\hline
         \end{tabular}
\end{table}

\subsection{2-D fractional PDEs}
Last, we consider two examples in two space dimension, with the domain $\Omega=(0,1)^2$ and $T=1$, and
perform numerical test on the following data
\begin{enumerate}
  \item[(e)] $f(x,y,t)= x(1-x)y(1-y)\sin t$ is smooth in time, with $\sin t \in \widetilde{H}_L^1(0,T)$.
  \item[(f)] $f(x,y,t)=x(1-x)y(1-y)t^{-0.3}$ is nonsmooth in time, with $t^{-0.3}\in \widetilde{H}_L^s(0,T)$, $s<0.2$.
\end{enumerate}

In either case, the source term is compatible with the zero initial data. In the computation, we
first divide the unit interval $(0,1)$ into $M$ equally spaced subintervals with a mesh size
$h = 1/M$, which partitions the domain $\Omega$ into $M^2$ small squares. Then we obtain regular partition of the domain
by connecting the diagonals. The results for cases (e) and (f) are given in Table \ref{tab:exam2D},
where the spatial mesh size $h$ is fixed at $h=1/100$ and the reference solution is computed with $\tau=1/2000$.
In case (e), the source $f$ is smooth and compatible, and the empirical convergence agrees well
with the theoretical prediction. In case (f), $f$ is nonsmooth in time, and the convergence
for small fractional order $\alpha$ suffers some loss, similar to the one-dimensional case.
\begin{table}[hbt!]
\centering
\caption{The relative error $\|u - u_{h \tau}\|_{L^2(Q_T)}/\|u \|_{L^2(Q_T)}$
for examples (e) and (f) with $\alpha=0.3$, $0.5$, $0.7$ and $0.9$, and $h=1/100$.}\label{tab:exam2D}
\begin{tabular}{|c|c|cccccc|c|}
	\hline
	case & $\alpha\backslash K $   &$10$ &$20$ & $40$ & $80$ &$160$ &$320$& rate\\
	\hline
	\multirow{3}{*}{(e)}
    & $0.3$  & 1.50e-2 & 6.15e-3 & 2.52e-3 & 1.05e-3 & 4.26e-4 & 1.75e-4 & 1.28 (1.30) \\
	&$0.5$   & 8.38e-3 & 3.06e-3 & 1.10e-3 & 4.02e-4 & 1.41e-4 & 5.05e-5 & 1.47 (1.50)\\
	&$0.7$   & 5.65e-3 & 1.88e-3 & 6.00e-4 & 1.85e-4 & 5.54e-5 & 1.67e-5 & 1.68 (1.70) \\
	&$0.9$   & 4.46e-3 & 1.30e-3 & 3.52e-4 & 9.28e-5 & 2.41e-5 & 6.31e-6 & 1.89 (1.90) \\
          \cline{1-9}
	\multirow{3}{*}{(f)}
    &$0.3$   & 3.31e-1 & 2.78e-1 & 2.31e-1 & 1.91e-1 & 1.54e-1 & 1.22e-1 & 0.28 (0.50)\\
	&$0.5$   & 3.15e-1 & 2.39e-1 & 1.77e-1 & 1.27e-1 & 8.74e-2 & 5.90e-2 & 0.48 (0.70)\\
	&$0.7$   & 2.76e-1 & 1.73e-1 & 1.01e-1 & 5.60e-2 & 2.98e-2 & 1.58e-2 & 0.82 (0.90)\\
    &$0.9$   & 2.06e-1 & 9.81e-2 & 4.37e-2 & 1.92e-2 & 8.50e-3 & 3.81e-3 & 1.15 (1.10)\\
	\hline
\end{tabular}
\end{table}

\section{Concluding Remarks}\label{sec:conclusions}
We have explored the viability of space-time discretizations for numerically solving
time-dependent fractional-order differential equations, and proposed a novel
Petrov-Galerkin finite element method on the pair of spaces ${\mathbb X}_h \times  \bVt $
as the trial space and ${\mathbb X}_h \times  {\mathbb W}_\tau $
as the test space, where the space $ \bVt$
consists of fractionalized piecewise constant
functions. One distinct feature of our approach
is that it leads to an unconditionally  stable time stepping scheme.
It may have interesting applications to
other types of fractional-order differential equations.\\

\section*{Acknowledgements}
B. Duan is supported by China Scholarship Council and the Fundamental Research Funds for
the Central Universities of Central South University (2016zzts015). The work of B. Jin is
partially supported by UK EPSRC grant EP/M025160/1, that of R. Lazarov by grant NSF-DMS 1620318
and that of Z. Zhou by the AFOSR MURI Center for Material Failure Prediction through
Peridynamics and the ARO MURI Grant W911NF-15-1-0562.\\

\bibliographystyle{abbrv}
\bibliography{references}

\end{document}